\documentclass[10pt]{amsart} 
\usepackage{latexsym}
\usepackage{subfigure}
\usepackage{array}
\usepackage{indentfirst}
\usepackage{fancyhdr}
\usepackage{epsf}
\usepackage{amsmath,bm}
\usepackage{amsthm}
\usepackage{booktabs}
\usepackage{amsfonts}
\usepackage{epsf}
\usepackage{multicol}
\usepackage[square, comma, sort&compress, numbers]{natbib}
\usepackage{fancyhdr}
\usepackage{pxfonts}
\usepackage{caption}     
\usepackage{setspace}
\usepackage{geometry}

\newtheorem{remark}{Remark}[section]

\newtheorem{theorem}{Theorem}[section]
\newtheorem{lemma}[theorem]{Lemma}

\newtheorem{assumption}{Assumption}[section]
\theoremstyle{definition}
\usepackage[titletoc]{appendix}
\numberwithin{equation}{section}

\usepackage{lipsum}
\usepackage{amsfonts}
\usepackage{graphicx}
\usepackage{epstopdf}
\usepackage{algorithmic}
\usepackage{hyperref}
\usepackage{float} 
\usepackage{amssymb}
\usepackage{stmaryrd}
\usepackage{tikz}
\usetikzlibrary {math}
\usetikzlibrary{patterns}
\usepackage{arydshln}

\newcommand{\bn}{\boldsymbol{n}}

\newcommand{\Om}{{\Omega} }
\newcommand{\Oml}{{\Omega^{+}} }
\newcommand{\Omr}{{\Omega^{-}} }
\newcommand{\Omlh}{{\Omega^{+}_h} }
\newcommand{\Omrh}{{\Omega^{-}_h} }

\newcommand{\cOmlh}{{\hat{\Omega}^{+}_h} }
\newcommand{\cOmrh}{{\hat{\Omega}^{-}_h} } 

\newcommand{\betal}{{\beta^{+}} }
\newcommand{\betar}{{\beta^{-}} }

\newcommand{\Vhl}{V_h^+} 
\newcommand{\Vhlg}{V_{g,h}^+} 
\newcommand{\Vhr}{V_h^-} 
\newcommand{\Qhr}{Q_h^-} 
\newcommand{\VP}{V_{h}^-} 
\newcommand{\VL}{V_h^+} 

\newcommand{\VLg}{V_{\rm g, P_1}^+} 
\newcommand{\PiL}{\Pi_{\rm P_1}}

\newcommand{\VRT}{Q_h^-} 
\newcommand{\PiRT}{\Pi_{\rm RT}}  
\newcommand{\PiRTnew}{\Pi_{\rm RT}^*}

\newcommand{\Vr}{V^-}
\newcommand{\Vl}{V^+}

\newcommand{\Hdiv}{Q^-}

\newcommand{\bmx}{\boldsymbol{x}}
\newcommand{\bmn}{\boldsymbol{n}}

\newcommand{\bmp}{v}

\newcommand{\fr}{f^-} 
\newcommand{\fl}{f^+}

\newcommand{\ul}{u^+}
\newcommand{\vl}{v^+} 
\newcommand{\ur}{u^-}
\newcommand{\vr}{v^-} 
\newcommand{\ulh}{u_h^+}
\newcommand{\vlh}{v_h^+} 
\newcommand{\urh}{u_h^-}
\newcommand{\vrh}{v_h^-} 
\newcommand{\urhp}{\hat u_h^-}
\newcommand{\taur}{\tau^-}
\newcommand{\sigmar}{\sigma^-}
\newcommand{\sigmal}{\sigma^+}
\newcommand{\tauh}{\tau_h^-}
\newcommand{\sigmah}{\sigma_h^-}

\newcommand{\Pih}{\Pi_h^0}

\newcommand{\dx}{\,{\rm dx}}
\newcommand{\ds}{\,{\rm ds}}

\newcommand{\neout}{{n_o}}

\newcommand{\Th}{\mathcal{T}_{h}}

\newcommand{\G}{\Gamma}
\newcommand{\Gh}{{\Gamma_h}}

\newcommand{\cT}{\mathcal{T}}

\newcommand{\Div}{{\nabla\cdot\,}}

\begin{document}

\title{A direct finite element method for  elliptic interface problems} 

\author{Jun Hu}
\address{LMAM and School of Mathematical Sciences, Peking University, Beijing 100871, People's Republic of China. hujun@math.pku.edu.cn} 

\author{Limin Ma}
\address{School of Mathematics and Statistics, Wuhan University, Wuhan  430072, People's Republic of China. limin18@whu.edu.cn} 
 
\maketitle
\begin{abstract}
In this paper, a direct finite element method is proposed for solving interface problems on unfitted meshes. This new method treats the two interface conditions as an $H^{\frac12}(\G)\times H^{-\frac12}(\G)$ pair for the mutual interaction across the interface, rather than the jumps of variables.  A simple and straightforward finite element method is proposed based on this approach. This method solves the interface problem using conforming finite elements in one subdomain and conforming mixed finite elements in the other, with a natural integral term accounting for mutual interaction.  Under reasonable assumptions, this direct finite element method is proved to be well-posed with an optimal a priori error analysis. Moreover, a simple lowest-order direct finite element method, using the linear element and the lowest-order Raviart-Thomas element, is analyzed to achieve the optimal a priori error estimate by verifying the aforementioned assumptions. Numerical tests are provided to confirm the theoretical results and the effectiveness of the direct finite element method.

  \vskip 15pt


 \vskip 15pt

    \end{abstract}

\section{Introduction} 
Consider the following elliptic interface problem  
\begin{align}
\label{source} 
&- \nabla \cdot(\beta\nabla u)= f & \mbox{ in }\Omega=\Oml\cup \Omr,\\ 
\label{jumpcon}
&[u]=0,\quad \left[\beta\nabla u\cdot\bmn\right]=0 & \mbox{ across }\Gamma,\\
\label{bdcon}
&u= g ,   &\mbox{ on }\partial\Omega,	
\end{align}
where $\Om\subset \mathbb{R}^2$ is a bounded Lipschitz domain, $f\in L^2(\Om)$, $g\in H^{\frac12}(\partial \Omega)$, and $\Gamma=\partial \Oml\cap \partial\Omr$ is a Lipschitz interface dividing $\Om$ into two non-intersecting subdomains $\Oml$ and $\Omr$. Here $\bmn$ denotes the unit outer normal to $\Omr$, and $[v]|_\G=v|_\Oml-v|_\Omr$ represents the jump of a function $v$ across the interface $\G$. The diffusion coefficient $\beta$  is assumed to be piecewise constant:
$$
\beta=\begin{cases}
\betal,&(x,y)\in \Oml\\
\betar,&(x,y)\in \Omr
\end{cases},\quad \mbox{with}\quad \min\{\betal,\betar\}>0.
$$
Interface problems with discontinuous coefficients frequently arise in material sciences and fluid dynamics, such as in the porous media equations in oil reservoirs. Numerical solutions of these interface problems have been extensively studied, see \cite{Li1998TheII,hansbo2002an,Hansbo2012ACF,Hansbo2004AFE,Hu2020FiniteEM,Xu2008UNIFORMCM,burman2018robust,Kergrene2016StableGF,Guzmn2015OnTA,Adjerid2015AnID,chen2017an,Zi2003NewCE,li2003new,Belytschko1999ElasticCG,babuska1970the,chen1998finite,xu1982estimate,guyomarch2009a,huynh2013a,barrett1987fitted,hu2021optimal}.

According to the topological relation between discrete grids and the interface, finite element methods for interface problems are divided into two main categories: interface-fitted finite element methods and interface-unfitted finite element methods.
Numerical methods using body-fitted meshes are well-studied for various interface problems and achieve optimal or nearly optimal convergence rates for arbitrarily shaped interfaces, as discussed in  \cite{babuska1970the,barrett1987fitted,chen2017an,chen1998finite,guyomarch2009a,Hu2020FiniteEM,hu2021optimal,huynh2013a,Xu2016OptimalFE,wieners1998coupling,2013Weak,2025The,guo2019Improved}. 
On the other hand, various unfitted finite element methods, where elements are allowed to intersect the interface, have been developed to avoid the complexity of generating interface-fitted meshes. Most of these methods handle the interface conditions \eqref{jumpcon} as jump conditions of  the solution $u$ and the flux $\beta\frac{\partial u}{\partial \bmn}$. Broadly speaking, two main approaches are used to manage these jump conditions.  One approach modifies the basis functions on elements that intersect the interface to ensure a satisfaction of the jump conditions in an $H^1$ sense within the finite element solutions. This methodology is adopted in immersed finite element methods, which were first proposed in \cite{li1998immersed} and later widely studied in \cite{Adjerid2015AnID,chen2009the,gong2008immersed,gong2009immersed,Li1998TheII,li2006the,li2003new,2024A,guo2023Solving,ji2023immersed} and references therein, and also applied to virtual element methods \cite{wang2023a}. The other approach employs interior penalty or Nitsche’s trick in \cite{nitsche1971ber} to penalize the jump of double-valued functions across the interface in the $L^2$-norm. A typical application of this approach is the cut finite element method, which weakly enforces interface conditions by introducing penalty terms on interface elements where degrees of freedom are double defined  \cite{burman2018robust,burman2010fictitious,burman2012fictitious,2022An,2024Unfitted,hansbo2002an}. The cut finite element method is also highly compatible with other methods, for example discontinuous Galerkin methods \cite{casagrande2016DG,wang2013hybridizable}, adaptive techniques \cite{chen2021adaptive}, and mesh generation algorithms \cite{chen2023arbitrarily,chen2024arbitrarily}.


In this paper, a direct finite element method (DiFEM for short hereinafter) is proposed for solving interface problems on unfitted meshes based on a coupled weak formulation. This formulation  interprets the two interface conditions \eqref{jumpcon} as an $H^{\frac12}(\G)\times H^{-\frac12}(\G)$ pair,  rather than jump conditions of variables. The interaction between the two subdomains is conducted by taking the value of~$u$ in $\Oml$ as a Dirichlet boundary condition for the equation in $\Omr$, and the value of $\beta \frac{\partial u}{\partial \bmn}$ in $\Omr$ as a Neumann boundary condition for the equation in $\Oml$. DiFEM solves the problem on $\Oml$ by conforming finite elements, and the problem on $\Omr$ by conforming mixed finite elements, where both elements are defined on elements crossing the interface. This requires the application quadrature rules and leads to consistency error, but avoids the modification of basis function spaces and also the penalization on jumps. Under reasonable assumptions regarding the discrete spaces and quadrature rules, DiFEM results in a well-posed symmetric saddle point system, and the well-posedness and an optimal a priori analysis are analyzed. Moreover, DiFEM can be applied directly to interface problems with non-homogeneous interface conditions.

A simple lowest-order DiFEM is proposed by using the linear element  for the primal form and the lowest-order Raviart-Thomas element for the mixed form.  
To ensure the critical inf-sup condition, a new Raviart-Thomas interpolation is designed by modifying the degrees of freedom on edges that are not interior to \( \Omega^- \). This new interpolation is proved to be bounded and  preserve the critical commuting property even on intersecting elements. By verifying the assumptions mentioned above, the well-posedness and an optimal a priori analysis for the lowest-order DiFEM are proved. 


The rest of this paper is organized as follows. Later in this section, we introduce the necessary notations and preliminaries. In Section \ref{sec:DiFEM}, the DiFEM is proposed, and proved to be well-posed and admit optimal convergence under certain assumptions. In Section \ref{sec:linear}, a simple lowest-order DiFEM is proposed, and an optimal a priori error estimate is analyzed. Finally, in Section \ref{sec:numerical}, numerical experiments are provided to verify the theoretical results and demonstrate the effectiveness of the proposed method.

Given a nonnegative integer $k$ and a bounded region $G\subset \mathbb{R}^2$, let  $H^k(G,\mathbb{R})$, $\| \cdot \|_{k,G}$, $|\cdot |_{k,G}$ and $(\cdot, \cdot)_G$  denote the usual Sobolev spaces, norm, semi-norm, and  the standard $L^2$ inner product over region $G$, respectively. For any curve $C$, let  $\langle \cdot, \cdot\rangle_C$ be the duality between $H^\frac12(C)$ and  $H^{-\frac12}(C)$, and reduce to the integral when the functions are piecewise polynomials. Let 
$$
H^k(\Oml\cup \Omr)=\{u: u|_{\Oml}\in H^k(\Oml),\quad  u|_{\Omr}\in H^k(\Omr)\}.
$$
For any given function $v$ on $\Om$, add a superscript `+' or `-' to represent the restriction of $v$ to $\Oml$ or $\Omr$, respectively, that is 
$
\vl=v|_\Oml$, $\vr=v|_\Omr.
$
By the extension theorem for Sobolev spaces, there exists $\tilde v^+\in H^2(\Om)$ such that  $\tilde v^+|_\Oml=\vl$ and $\|\tilde v^+\|_{1,\Om}\lesssim \|\vl\|_{1,\Oml}$ suppose that $\vl\in H^1(\Oml)$. We can also define $\tilde u^-$, $\tilde \sigma^-$ and $\tilde f^-$ in the same way. For the ease of presentation, we will omit the tilde  of all these variables in this paper.

Suppose that $\Om$ is a convex polygonal domain  in $\mathbb{R}^2$. For a triangulation $\Th$ of domain $\Om$, let $|K|$  and $h_K$ be the area and the diameter of an element $K\in\Th$, respectively, and  $h=\max_{K\in\Th}h_K$.  For each element $K\in \Th$, define
$$
K^+=K\cap \Oml,\qquad K^-=K\cap \Omr,\qquad \G_K=K\cap \G.
$$
For $K\subset\mathbb{R}^2$ and $r\in \mathbb{Z}^+$, let $P_r(K, \mathbb{R})$ be the space of all polynomials of degree not greater than $r$ on~$K$. Throughout the paper, a positive constant independent of the mesh size is denoted by $C$, which refers to different values at different places. For ease of presentation, we shall use the symbol $A\lesssim B$ to denote that $A\leq CB$.

\section{Direct Finite Element Method for interface problem}\label{sec:DiFEM}
In this section, we propose the DiFEM for \eqref{source} based on a weak formulation coupling the primal form and the mixed form, and also analyze the well  well-posedness of  the DiFEM under reasonable assumptions.

\subsection{A weak formulation of interface problem}\label{sec:weak}

Define the spaces on $\Oml$
$$
\Vl_g = \{\ul \in H^1(\Oml), \ul=g \ \mbox{on}\ \partial\Om\cap \partial\Oml\},\quad 
\Vl = \{\ul \in H^1(\Oml), \ul=0 \ \mbox{on}\ \partial\Om\cap \partial\Oml\},
$$
and two Sobolev spaces on $\Omr$
$$
\Vr = L^2(\Omr),\quad
\Hdiv=\{\taur\in L^2(\Omr): \Div \taur\in L^2(\Omr)\}.
$$
Let $\sigmal=\betal \Div \ul$ on the region $\Oml$ and $\sigmar=\betar\Div \ur$ on the region $\Omr$. 
The primal formulation is adopted for the second order elliptic equation \eqref{source} on $\Oml$, that is for any $\vl\in \Vl$, 
\begin{equation}\label{linearpart}
(\betal \nabla \ul, \nabla \vl)_\Oml =(\fl, \vl)_\Oml - \langle \betal\frac{\partial \ul }{\partial \bn}, \vl\rangle_\G=(\fl, \vl)_\Oml - \langle\sigmal\cdot\bn, \vl\rangle_\G,
\end{equation}  
where $\bmn$ is the unit normal pointing from $\Omr$ to $\Oml$.
The mixed formulation for the second order elliptic equation \eqref{source} on~$\Omr$ reads
\begin{equation} \label{mixpart}
\left\{
\begin{aligned}
\frac{1}{\betar}(\sigmar,\taur)_\Omr + (\ur,\Div\taur)_\Omr &=\langle \ur, \taur\cdot \bmn\rangle_\G
+\langle g, \taur\cdot \bmn\rangle_{\G_b},\quad \forall \taur\in \Hdiv
\\
(\Div \sigmar,\vr)_\Omr &=-(\fr,\vr)_\Omr,\ \qquad \quad\ \qquad \qquad \forall \vr\in \Vr,
\end{aligned}
\right.
\end{equation} 
where $\G_b=\partial\Om\cap \partial\Omr$. 
Note that  $\vl|_\G \in H^{\frac12}(\G)$ for any $\vl\in \Vl$ and $\taur\cdot\bn\in H^{-\frac12}(\G)$ for any $\taur\in \Hdiv$. The two interface conditions on  $u\in H^1(\Oml\cup \Omr)$ and $\sigma\cdot \bn$ with $\sigma\in {\rm H(div},\Oml\cup \Omr)$ form an $H^{\frac12}(\G)\times H^{-\frac12}(\G)$ pair for the interface integral $\langle\sigmal\cdot\bn, \vl\rangle_\G$ in the primal formulation \eqref{linearpart} and $\langle \ur, \taur\cdot \bmn\rangle_\G$ in the mixed formulation~\eqref{mixpart}. By the interface requirements in \eqref{source}, 
$$
\sigmal\cdot \bn=\sigmar\cdot \bn,\quad \ul=\ur.
$$
We can obtain the coupled formulation in \cite{wieners1998coupling} seeking $(\sigmar,\ur,\ul)\in \Hdiv\times \Vr\times \Vl_g$ such that for any $(\taur,\vr,\vl)\in \Hdiv\times \Vr\times \Vl$, 
\begin{equation}\label{weak0}
\left\{
\begin{aligned}
\frac{1}{\betar}(\sigmar,\taur)_\Omr + (\ur,\Div\taur)_\Omr -\langle \ul, \taur\cdot \bmn\rangle_\G&=\langle g, \taur\cdot \bmn\rangle_{\G_b},
\\
(\Div \sigmar,\vr)_\Omr \qquad \qquad \qquad \qquad \qquad \qquad \qquad \quad &=-(\fr,\vr)_\Omr,
\\
- \langle \sigmar\cdot \bmn, \vl\rangle_{\G}\qquad \qquad \qquad -(\betal \nabla \ul, \nabla \vl)_\Oml &=-(\fl, \vl)_\Oml,
\end{aligned}
\right.
\end{equation} 
which is a symmetric perturbed saddle point system. By substracting the last equation from the first one, the formulation \eqref{weak0} can be rewritten as an equivalent saddle point system
\begin{equation}\label{weak}
\left\{\begin{aligned}
a(\sigmar,\ul;\taur,\vl) + b(\ur;\taur, \vl)&=(\fl,\vl)_\Oml
+\langle g, \taur\cdot \bmn\rangle_{\G_b}
\\
b(\vr;\sigmar,\ul)&=(\fr,\vr)_\Omr,
\end{aligned}\right. 
\end{equation} 
where the bilinear forms  
\begin{equation}\label{abdef}
\begin{aligned}
a(\sigmar,\ul;\taur,\vl)=& \frac{1}{\betar}(\sigmar,\taur)_\Omr +(\betal \nabla \ul, \nabla \vl)_\Oml+ \langle  \sigmar\cdot \bmn, \vl\rangle_\G-\langle  \ul, \taur\cdot \bmn\rangle_\G,
\\
b(\ur;\taur, \vl)=&(\Div \taur,\ur)_\Omr.
\end{aligned}
\end{equation} 
Although \eqref{weak0} and \eqref{weak} are equivalent, the compact form \eqref{weak} is not symmetric because of the bilinear form $a(\cdot,\cdot)$ in \eqref{abdef}. 
We will analyze the well-posedness of the proposed weak formulation \eqref{weak0} in terms of the nonsymmetric compact form \eqref{weak}.
For any $(\taur,\vl)\in \Hdiv\times \Vl$ and $\vr\in \Vr$, define the norms
\begin{equation}\label{norm} 
\begin{aligned}
\interleave (\taur,\vl)\interleave_1=&\frac{1}{\sqrt{\betar}}\|\taur\|_{0,\Omr}  
+ \|\nabla\cdot \taur\|_{0,\Omr}
+ \sqrt{\betal}\|\nabla \vl\|_{0,\Oml},\quad
\interleave \vr\interleave_0 = \|\vr\|_{0,\Omr}.
\end{aligned}
\end{equation}  

The wellposedness result in the following lemma reveals how the solution of the coupled formulation depend on the coefficients $\betal$ and $\betar$.
\begin{lemma}\label{lm:wellposeweak}
The weak formulation \eqref{weak} is well defined, namely, there exists a unique solution $(\sigmar,\ur,\ul)\in \Hdiv\times \Vr\times \Vl_g$ of  \eqref{weak} and 
$$
\begin{aligned}
\interleave (\sigmar,\ul)\interleave_1\le &C(\frac{1}{\sqrt{\betal}}\|\fl\|_{0,\Oml} + \max(\sqrt{\betar},1)\|g\|_{\frac12,\G_b} + C_\beta\|\fr\|_{0,\Omr}),
\\
\interleave  \ur\interleave_0\le& CC_\beta(\frac{1}{\sqrt{\betal}}\|\fl\|_{0,\Oml} + \max(\sqrt{\betar},1)\|g\|_{\frac12,\G_b} +C_\beta\|f\|_{0,\Omr}),
\end{aligned}
$$
where positive constant $C$ is independent of $\betal$ and $\betar$, and $
C_\beta=\frac{\max(\sqrt{\betar/\betal},1)}{\min(\sqrt{\betar},1)}
$.

\end{lemma}
\begin{proof}
By the definition of the norms in \eqref{norm}, the two bilinear forms $a(\sigmar,\ul;\taur,\vl)$ and $b(\ur;\taur, \vl)$ are continuous. To be specific, there exists a positive constant $C$ such that
$$
\begin{aligned}
|a(\sigmar,\ul;\taur,\vl)|\le &C\max(\sqrt{\betar/\betal},1)\interleave (\sigmar,\ul)\interleave_1\interleave (\taur,\vl)\interleave_1,
\\
|b(\ur;\taur, \vl)|\le & \interleave (\taur,\vl)\interleave_1\interleave \ur\interleave_0,
 \end{aligned}
$$
where constant $C$ is independent of $\betal$ and $\betar$.
Define the kernel space of  $\Hdiv\times\Vl$ by
$$
Z=\{(\taur,\vl)\in \Hdiv\times\Vl: (\Div \taur,\vr)_\Omr=0,\quad \forall \vr\in \Vr\}.
$$
Note that $\Div\taur=0$ for any $(\taur,\vl)\in Z$, which indicates that the bilinear form $a(\sigmar,\ul;\taur,\vl)$ is coercive on $Z$, namely, 
\begin{equation}
a(\taur,\vl;\taur,\vl)= \frac{1}{\betar}\|\taur\|_{0,\Omr}^2 +\betal \|\nabla \vl\|_{0,\Oml}^2
\ge \frac12 \interleave (\taur,\vl)\interleave_1^2,\quad \forall (\taur,\vl)\in Z.
\end{equation}
Then, the uniqueness of the solution of \eqref{weak} is guaranteed by Theorem 4.2.1 in \cite{boffi2013mixed}.

For any $\vr\in \Vr$, there exists $\taur\in \Hdiv$ such that $\Div\taur = \vr$ and $\|\Div\taur\|_{0,\Omr} + \|\taur\|_{0,\Omr}\le C\|\vr\|_{0,\Omr}$. 
By the definition of $b(\ur;\taur, \vl)$, the inf-sup condition below holds
$$
\inf_{0\neq \vr\in \Vr}\sup_{(\taur,\vl)\in  \Hdiv\times \Vl}\frac{b(\vr;\taur, \vl)}{ \interleave (\taur,\vl)\interleave_1 \interleave \vr\interleave_0}
\ge \frac1C\inf_{0\neq \vr\in \Vr}\frac{\interleave \vr\interleave_0^2}{\max(\frac{1}{\sqrt{\betar}},1) \interleave \vr\interleave_0^2}
\ge \frac1C\min(\sqrt{\betar},1).
$$
It follows from Theorem 4.2.3 in \cite{boffi2013mixed} that 
$$
\begin{aligned}
\interleave (\sigmar,\ul)\interleave_1\le &C(\frac{1}{\sqrt{\betal}}\|\fl\|_{0,\Oml} + \max(\sqrt{\betar},1)\|g\|_{\frac12,\G_b} + C_\beta\|\fr\|_{0,\Omr}),
\\
\interleave  \ur\interleave_0\le& CC_\beta(\frac{1}{\sqrt{\betal}}\|\fl\|_{0,\Oml} + \max(\sqrt{\betar},1)\|g\|_{\frac12,\G_b} +C_\beta\|\fr\|_{0,\Omr}),
\end{aligned}
$$
which completes the proof. 
\end{proof}

\subsection{The direct finite element method}\label{sec:direct}
Given a triangulation $\cT_h$ of the whole domain $\Om$, we call  an element to be a non-interface element if the interface does not intersect with the element, or the interface intersects an edge  only at its vertices or this whole edge lies on the interface, and to be a non-interface element if the element is not an interface element. For ease of presentation, we denote the union of all non-interface elements in $\Oml$ and all interface elements by $\cT_h^+$, and the union of all non-interface elements in $\Omr$ and all interface elements by $\cT_h^-$.

We employ a conforming finite element in $H^1(\Oml)$ with shape function space $S_u^+(K)$ for $\ul$, a conforming finite element in $H(\mbox{div},\Omr)$ with shape function space $S_\sigma^-(K)$ for $\sigmar$, and an element in $L^2(\Omr)$ with shape function space $S_u^-(K)$ for $\ur$. 
Define the shape function space of the finite element for $\urh$ by $S^-_u(K)$, and
Define the corresponding finite element spaces 
\begin{equation}\label{conformspaces}
\begin{aligned}
\Vhl=&\{\vlh\in H^1(\Oml): \vlh|_{K}\in  S_u^+(K),\ \mbox{where}\ K\in \cT_h^+, \ \vlh|_{\partial \Om\cap \partial \Oml}=0\},
\\
\Qhr=&\{\tauh\in H(\mbox{div},\Omr): \tauh|_{K}\in  S_\sigma^-(K),\ \mbox{where}\ K^-\in \cT_h^-\},
\\
\Vhr=&\{\vrh\in L^2(\Omr): \vlh|_{K}\in  S_u^-(K),\ \mbox{where}\ K^-\in \cT_h^-\},
\end{aligned}
\end{equation} 
and $ \Vhlg$ is the finite element space for $\ulh$ with Dirichlet boundary condition $g$. Quadrature schemes are required for computing inner products over the subdomains $\Omr$ or $\Oml$, as well as along the interface $\G$ or boundary $\G_b$.  Let
$$
 (\cdot,\cdot)_{K^+,h},\qquad  (\cdot,\cdot)_{K^-,h},\qquad \langle\cdot,\cdot,\rangle_{\G_K,h}
 $$ 
be the discrete inner products obtained by empolying some quadrature schemes. Then, we can define discrete inner products $
\langle \cdot,\cdot\rangle_{\G,h}=\sum_{K:K_\G\neq \emptyset} \langle \cdot,\cdot\rangle_{\G_K,h}$, and  
$
(\cdot,\cdot)_{\Om^s,h}=\sum_{K:K^s\neq \emptyset} (\cdot,\cdot)_{K^s,h}$
with $s=+$, $-$.

Equipped with the conforming finite spaces  and quadrature schemes defined above, we propose
the DiFEM seeking $(\sigmah,\urh,\ulh)\in\Qhr \times \Vhr\times \Vhlg$ such that for any $(\tauh,\vrh,\vlh)\in \Qhr\times \Vhr\times \Vhl$, 
\begin{equation}\label{discrete}
\left\{\begin{aligned}
a_h(\sigmah,\ulh;\tauh,\vlh) + b_h(\urh;\tauh, \vlh)&=(\fl,\vlh)_{\Oml,h}
+\langle g, \tauh\cdot \bmn\rangle_{\G_b,h}
\\
b_h(\vrh;\sigmah,\ulh)&=(\fr,\vrh)_{\Omr,h},
\end{aligned}\right. 
\end{equation}
where the bilinear forms  
\begin{equation}\label{abquadrature}
\begin{aligned}
a_h(\sigmah,\ulh;\tauh,\vlh)=& \frac{1}{\betar}(\sigmah,\tauh)_{\Omr,h} +(\betal \nabla \ulh, \nabla \vlh)_{\Oml,h}+ \langle\sigmah\cdot \bmn, \vlh\rangle_{\G,h}-  \langle \tauh\cdot \bmn, \ulh\rangle_{\G,h},
\\
b_h(\urh;\tauh, \vlh)=&(\Div \tauh,\urh)_{\Omr,h}. 
\end{aligned}
\end{equation} 
By the definition in \eqref{conformspaces}, approximation $\ulh$ is defined on elements in $\cT_h^+$, both $\urh$ and $\sigmah$ are defined on elements in $\cT_h^-$, and all the variables $\ulh$, $\urh$, $\sigmah$ are defined on interface elements.
Similar to the weak formulation in \eqref{weak0}, the discrete formulation~\eqref{discrete} can be rewritten as a symmetric perturbed saddle point system.

Various finite element spaces and quadrature rules can be applied to  \eqref{discrete}, resulting in various discrete schemes.
Some assumptions on discrete spaces and quadrature schemes are proposed below to guarantee the wellposedness and optimal error analysis of the proposed DiFEM~\eqref{discrete}. 


\begin{assumption}\label{ass:error}
\begin{enumerate}
\item[]
\item[(A1)] [Boundedness] The bilinear forms are bounded with respect to the norms. To be specific,
\begin{equation}\label{bdd}
|a_h(\sigmah,\ulh;\tauh,\vlh)|\lesssim \interleave (\sigmah,\ulh)\interleave_{1,h} \interleave (\tauh,\vlh)\interleave_{1,h} ,
\quad
|b_h(\vrh;\sigmah,\ulh)|\lesssim   \interleave \vrh\interleave_{0,h}\interleave (\sigmah,\ulh)\interleave_{1,h}.
\end{equation}   
where $\|w_h\|_{0,\Om^s}\lesssim \|w_h\|_{\Om^s,h}=\sqrt{(w_h,w_h)_{\Om^s,h}}\lesssim \|w_h\|_{0,\Om^s}$ with $s=+$ or $-$ and
\begin{equation}\label{normh} 
\begin{aligned}
\interleave (\tauh,\vlh)\interleave_{1,h}^2 =&\frac{1}{\betar}\|\tauh\|_{\Omr,h}^2  
+ \|\nabla\cdot \tauh\|_{\Omr,h}^2
+ \betal\|\nabla \vlh\|_{\Oml,h}^2,
\quad
 \interleave \vrh\interleave_{0,h}=\|\vrh\|_{\Omr,h}.
\end{aligned}
\end{equation}    
\item[(A2)] 
[Inf-sup condition] The following inf-sup condition holds 
\begin{equation}\label{infsupb}
\inf_{0\neq \vrh\in \Vhr}\sup_{(\tauh,\vlh)\in \Qhr\times\Vhl}\frac{b_h(\vrh;\tauh, \vlh)}{\interleave (\tauh,\vlh)\interleave_{1,h} \interleave\vrh \interleave_{0,h} }\ge \alpha>0.
\end{equation}   
\item[(A3)][Approximation] The discrete spaces with $ \nabla\cdot \Qhr\subset \Vhr$ admit the following approximation property
\begin{equation}\label{app}
\inf_{(\tauh, \vlh)\in \Qhr\times \Vhl}\interleave (\sigmar-\tauh,\ul-\vlh)\interleave_{1,h} 
+ \inf_{\vrh\in \Vhr}\interleave\ur-\vrh\interleave_{0,h}\lesssim h^k,
\end{equation} 
provided that $u\in H^{k+1}(\Oml\cup\Omr)\cap H^1(\Om)$.
\item[(A4)][Quadrature accuracy] Consistency error of the DiFEM~\eqref{discrete} satisfies  that
\begin{equation}\label{consistency}
\begin{aligned} 
&\sup_{(\tauh,\vlh)\in \Qhr\times \Vhl}\frac{|a(\sigmar,\ul;\tauh,\vlh) - a_h(\sigmar,\ul;\tauh,\vlh)|}{\interleave (\tauh,\vlh)\interleave_{1,h} } \lesssim h^k,
\\
&\sup_{(\tauh,\vlh)\in \Qhr\times \Vhl}\frac{ |b(\ur;\tauh,\vlh)  - b_h(\ur;\tauh,\vlh) |}{\interleave (\tauh,\vlh)\interleave_{1,h} }
+ \sup_{\vrh\in \Vhr}\frac{ |b(\vrh;\sigmar,\ul)  - b_h(\vrh;\sigmar,\ul) |}{\interleave \vrh\interleave_{0,h} } \lesssim h^k,
\\
&\sup_{(\tauh,\vlh)\in  \Qhr\times \Vhl}\frac{|(f^+, v_h^+)_{\Om^+} - (f^+, v_h^+)_{\Om^+,h}| }{\interleave (\tauh,\vlh)\interleave_{1,h} } 
+ 
\sup_{\vrh\in \Vhr}\frac{|(f^-, v_h^-)_{\Om^-} - (f^-, v_h^-)_{\Om^-,h}|}{\interleave \vrh \interleave_{0,h} } \lesssim h^k,
\\
&\sup_{(\tauh,\vlh)\in \Qhr\times \Vhl}\frac{|\langle g, \tauh\cdot \bmn\rangle_{\G_b} - \langle g, \tauh\cdot \bmn \rangle_{\G_b,h}|}{\interleave (\tauh,\vlh)\interleave_{1,h} } \lesssim h^k,
\end{aligned}
\end{equation}
provided that $u\in H^{k+1}(\Oml\cup \Omr)\cap H^1(\Om)$.
\end{enumerate}
\end{assumption} 
Note that the assumptions above are standard requirements for the well-posedness and the convergence of finite element methods.
By the classic Babuška–Brezzi theory in \cite{boffi2013mixed},  a  discrete version of the analysis  in Lemma \ref{lm:wellposeweak} leads to the following well-posedness and optimal a priori analysis of the DiFEM \eqref{discrete} under Assumption~\ref{ass:error}.

\begin{theorem}\label{th:wellposedirect}
Under Assumption \ref{ass:error},  the proposed DiFEM  \eqref{discrete} is well defined, namely, there exists a unique solution $(\sigmah,\urh,\ulh)\in \Qhr\times \Vhr\times \Vhlg$ of  \eqref{discrete} and 
$$
\interleave (\sigmah,\ulh)\interleave_{1,h}  + \interleave  \urh\interleave_{0,h} \lesssim \|f\|_{0,\Om} + \|g\|_{\frac12, \G_b}.
$$ 
Moreover,  the solution $(\sigmah,\urh,\ulh)$ admits the optimal convergence
\begin{equation}\label{energyest}
\interleave (\sigmar-\sigmah, \ul-\ulh)\interleave_{1,h}  + \interleave \ur- \urh\interleave_{0,h} \lesssim h^k,
\end{equation}
provided that $u\in H^{k+1}(\Oml\cup\Omr)\cap H^1(\Om)$.
\end{theorem}
\begin{proof}
Define the discrete kernel space
$$
Z_h=\{(\tauh,\vlh)\in\Qhr\times \Vhl: (\Div \tauh,\vrh)_{\Omr,h}=0,\quad \forall \vrh\in \Vhr\},
$$
and 
$$
Z_h^B=\{(\tauh,\vlh)\in\Qhr\times \Vhl: (\Div \tauh,\vrh)_{\Omr,h}=(\fr,\vrh)_{\Omr,h},\quad \forall \vrh\in \Vhr\},
$$
For any $(\tauh,\vlh)\in Z_h\backslash \{0\}$, it follows from \eqref{abquadrature}, Assumption $(A1)$ and $(A2)$ that $\nabla\cdot \tauh=0$ and
\begin{equation}\label{infsupa}
a_h(\tauh,\vlh;\tauh,\vlh) = \frac{1}{\betar}\|\tauh\|_{\Omr,h}^2 + \betal \|\nabla \vlh\|_{\Oml,h}^2=\interleave (\tauh,\vlh)\interleave_{1,h} ^2,
\end{equation}
which indicates that $a_h(\cdot,\cdot)$ is coercive in  $Z_h$.
A combination of the boundedness \eqref{bdd}, the inf-sup conditions \eqref{infsupb} of  $b_h(\cdot, \cdot)$, the coercivity \eqref{infsupa} of  $a_h(\cdot,\cdot)$ in $Z_h$,  and the classic Babuška–Brezzi theory in \cite{boffi2013mixed} proves 
that the DiFEM   \eqref{discrete} is well posed.

For any $(\xi_h^-,w_h^+)\in Z_h^B$, it holds that $(\tauh,\vlh)=(\sigmah-\xi_h^-,\ulh-w_h^+)\in Z_h$. Note that
$$
a_h(\sigmah,\ulh;\tauh,\vlh)=(\fl,\vlh)_{\Oml,h}
+\langle g, \tauh\cdot \bmn\rangle_{\G_b,h},\qquad \forall (\tauh,\vrh)\in Z_h.
$$
It follows from the coercivity  of $a_h(\cdot,\cdot)$ on $Z_h$  that
$$
\begin{aligned}
\interleave (\sigmah-\xi_h^-,\ulh-w_h^+)\interleave_{1,h}^2= 
& a_h(\sigmar-\xi_h^-,\ul-w_h^+;\tauh,\vlh)
- \left(a_h(\sigmar,\ul;\tauh,\vlh) - a(\sigmar,\ul;\tauh,\vlh)\right)
\\
&- \left(a(\sigmar,\ul;\tauh,\vlh) - a_h(\sigmah,\ulh;\tauh,\vlh)\right).
\end{aligned}
$$
The fact $(\tauh,\vlh)\in Z_h$ implies that $b_h(\urh-\ur;\tauh, \vlh) =b_h(\vrh-\ur;\tauh, \vlh) $ for any $\vrh\in \Vhr$.
By the equations \eqref{weak} and \eqref{discrete} and ,
$$
\begin{aligned}
a(\sigmar,\ul;\tauh,\vlh) - a_h(\sigmah,\ulh;\tauh,\vlh)
=&\left((\fl,\vl)_\Oml
-(\fl,\vlh)_{\Oml,h}\right)
+\left(\langle g, \taur\cdot \bmn\rangle_{\G_b}
-\langle g, \tauh\cdot \bmn\rangle_{\G_b,h}\right)
\\
&+ \left(b_h(\ur;\tauh, \vlh)- b(\ur;\tauh, \vlh) \right)
+ b_h(\vrh-\ur;\tauh, \vlh).
\end{aligned}
$$
A combination of the boundedness of $a_h(\cdot,\cdot)$ and $b_h(\cdot,\cdot)$ in Assumption $(A1)$, Assumptions $(A4)$, and the two estimates above leads to 
$$
\interleave (\sigmah-\xi_h^-,\ulh-w_h^+)\interleave_{1,h}\lesssim \inf_{(\xi_h^-,w_h^+)\in Z_h^B}\interleave (\sigmar-\xi_h^-,\ul-w_h^+)\interleave_{1,h} + \inf_{\vrh\in \Vhr}\interleave\ur-\vrh\interleave_{0,h} + h^k,
$$
which indicates that 
$$
\interleave (\sigmar-\sigmah,\ul-\ulh)\interleave_{1,h}\lesssim \inf_{(\xi_h^-,w_h^+)\in Z_h^B}\interleave (\sigmar-\xi_h^-,\ul-w_h^+)\interleave_{1,h} + h^k.
$$
For any $(\tauh,\vlh)\in \Qhr\times \Vhl$, it follows from the inf-sup condition \eqref{infsupb} that
$$
\interleave (\sigmar-\xi_h^-,\ul-w_h^+)\interleave_{1,h}\lesssim  \interleave (\sigmar-\tauh,\ul-\vlh)\interleave_{1,h} + \sup_{\vrh\in \Vhr}{b_h(\vrh;\tauh-\xi_h^-,\vlh-w_h^+)\over \interleave \vrh\interleave_{0,h}}.
$$
By equation \eqref{weak} and the definition of $Z_h^B$,
$$
\begin{aligned}
b_h(\vrh;\tauh-\xi_h^-,\vlh-w_h^+)=&b_h(\vrh;\tauh-\sigmar,\vlh-\ul) 
+ \left(b_h(\vrh;\sigmar,\ul) - b(\vrh;\sigmar,\ul)\right)
\\
&+\left( (\fr,\vrh)_{\Omr} - (\fr,\vrh)_{\Omr,h}\right).
\end{aligned}
$$
A combination of the boundedness of $b_h(\cdot,\cdot)$ in Assumption $(A1)$, Assumptions $(A3)$-$(A4)$, and the three estimates above leads to 
$$
\interleave (\sigmar-\sigmah,\ul-\ulh)\interleave_{1,h}\lesssim   h^k.
$$  

By equations \eqref{weak} and \eqref{discrete},
\[
\begin{aligned}
b_h(\ur-\urh;\tauh,\vlh)=&\left(b_h(\ur;\tauh,\vlh) - b(\ur;\tauh,\vlh)\right)
+\left((\fl,\vlh)_{\Oml} - (\fl,\vlh)_{\Oml,h}\right)
\\
&
+\left(\langle g, \tauh\cdot \bmn\rangle_{\G_b} - \langle g, \tauh\cdot \bmn\rangle_{\G_b,h} \right)
- \left(a(\sigmar,\ul;\tauh,\vlh) - a_h(\sigmar,\ul;\tauh,\vlh)\right)
\\
&- a_h(\sigmar-\sigmah,\ul-\ulh;\tauh,\vlh).
\end{aligned}
\]
A combination of the boundedness in Assumption $(A1)$, Assumptions $(A2)$-$(A4)$, and the  estimate above leads to 
$$
\interleave \urh - \vrh \interleave_{0,h}\lesssim \sup_{(\tauh,\vlh)\in \Qhr\times \Vhl}
{|b_h(\ur-\vrh;\tauh,\vlh) | + | b_h(\ur-\urh;\tauh,\vlh)|\over \interleave (\tauh,\vlh) \interleave_{1,h}}\lesssim h^k,
$$
which implies \eqref{energyest} and completes the proof.

\end{proof}

\begin{remark}
The direct finite element method can be generalized to solve interface problems with nonhomogeneous interface condition 
$$
[u]=g_1,\quad \left[\beta\nabla u\cdot\bmn\right]=g_2  \mbox{ across }\Gamma,
$$
which leads to the discrete problem with the same bilinear forms $a_h(\cdot,\cdot)$ and $b_h(\cdot,\cdot)$ and slightly different right-hand sides. This indicates that the wellposedness of the discrete problem under this interface condition also holds if the discrete spaces and quadrature schemes satisfy the assumptions in Assumption \ref{ass:error}.

%
\end{remark}

\section{A simple lowest-order DiFEM and optimal error estimate}\label{sec:linear}
In this section, we consider the lowest-order finite element method with a particular quadrature formula, and analyze the well-posedness and optimal a priori error estimate.

Consider the lowest-order DiFEM, where the linear element is employed for $\ulh$ and the lowest order Raviart-Thomas element  for $\sigmah$, namely, the shape function spaces in \eqref{conformspaces} are
\begin{equation}\label{shapefunc}
S^+_u(K)=P_1(K,\mathbb{R}),\quad S^-_u(K)=P_0(K,\mathbb{R}),\quad S_\sigma^-(K)=P_0(K,\mathbb{R}^d) + \bmx P_0(K,\mathbb{R}).
\end{equation}
The quadrature schemes are required for the inner products in \eqref{abquadrature}. It needs to be mentioned that quadrature schemes on elements sharing nonempty and empty intersections with the interface~$\G$ are slightly different. For ease of presentation, given an interface element $K$, denote the line segment connecting the two intersects of $\G$ and~$\partial K$ by $\G_{K,h}$, the region enclosed by $\G_{K,h}$ and the subset of $\partial K$ inside $\Oml$ by $K_h^+$, and the remaining part of $K$ by $K_h^-$, namely $K=K_h^+\cup K_h^-$. Figure \ref{fig:interface} depicts a simple example for the notations here. We consider the following quadrature scheme. For any $\G_K$, 
\begin{equation} \label{quad11}
\begin{aligned}
\langle v, 1\rangle_{\G_K,h}=|\G_{K,h}|v(x_K^\G),\qquad x_K^\G\mbox{ is the midpoint of }\G_{K,h}
\end{aligned}
\end{equation}  
For any $G$ enclosed by three end-to-end curves, let $G_h$ be the triangle, where the vertices $\{\boldsymbol{p}_j\}$ are the endpoints of the curves, and 
\begin{equation} \label{quad12}
\begin{aligned}
(v, 1)_{G,h}=\frac13|G_h|\sum_{i=1}^3 v(x_{i,G}),\qquad x_{i,G}=\sum_{j=1}^3\lambda_{i,j}\boldsymbol{p}_j, 
\end{aligned}
\end{equation}  
where $\{(\lambda_{i,1}, \lambda_{i,2}, \lambda_{i,3})\}_{i=1}^3$ are $(2/3, 1/6, 1/6)$, $(1/6, 2/3, 1/6)$, and $(1/6, 2/3, 1/6)$, respectively. For any~$K^s$ enclosed by four end-to-end curves, since $K=K^s\cup (K/K^s)$, where both $K$ and $K/K^s$ are regions  enclosed by three end-to-end curves, let
\begin{equation} \label{quad13}
\begin{aligned}
(v, 1)_{K^s,h}=(v,1)_{K,h}-(v,1)_{K/ K^s,h}.
\end{aligned}
\end{equation}   
Note that the quadrature schemes \eqref{quad11}-\eqref{quad13} satisfy the condition that  
\begin{equation}\label{quadcondition}
\langle v, 1\rangle_{\G_{K},h} =\int_{\G_{K,h}}v\ds,\quad (w, 1)_{K^s, h}= \int_{K_h^s}w\dx,\qquad \forall v\in P_1(K,\mathbb{R}),\ w\in P_2(K,\mathbb{R}).
\end{equation} 

\begin{figure}[!ht]
\begin{center}
\begin{tikzpicture}[xscale=2,yscale=2]
\draw[-] (0,0) -- (1.5,-0.1);
\draw[-] (1.5,-0.1) -- (0.7,1.2);
\draw[-] (0.7,1.2) -- (0,0);  
\draw[-] (1.5,-0.1) -- (2.3,1.1);
\draw[-] (0.7,1.2) -- (2.3,1.1); 
\draw[-] (0.5,1.2) parabola bend (1.2,0.2) (2,0); 
\node at (0.65,1.3) {$ \Oml$};  
\node at (0.1,0.6) {$ \Omr$};   
\draw[ultra thick] (0.58,0.98) -- (1.32,0.2);
\draw[ultra thick] (1.32,0.2) -- (1.66,0.14); 
\node at (0.8,0.5) {$ \G_{K_1,h}$};  
\node at (1.5,0.3) {$ \G_{K_2,h}$};   

\end{tikzpicture} 
\end{center}
\caption{The thick solid lines denote the approximate interface $\G_{K,h}$.
}
\label{fig:interface}
\end{figure}

Next, we begin to prove that the DiFEM \eqref{discrete} equipped with quadrature schemes \eqref{quad11}-\eqref{quad13}  in conforming finite spaces \eqref{conformspaces} with  \eqref{shapefunc} is wellposed and admits the optimal convergence. According to Theorem \ref{th:wellposedirect}, it only remains to verify the assumptions in  Assumption \ref{ass:error}.


\begin{lemma}\label{lm:bdd}
The bilinear forms $a_h(\cdot,\cdot)$ and $b_h(\cdot,\cdot)$ in \eqref{abquadrature} with the particular quadrature formulas \eqref{quad11}-\eqref{quad13}   are bounded with respect to the norms in~\eqref{normh}.
\end{lemma}
\begin{proof}
We first prove that $\|\nabla \vlh\|_{\Om^+,h}$ and $\|\tauh\|_{\Om^-,h}$ are norms in $\VL$ and $\VRT$, respectively. Thanks to \eqref{quadcondition}, it is easy to verify the semipositive definite property, the linearity and the triangle inequality. If  $\|\nabla \vlh\|_{\Om^+,h}=0$, by \eqref{quadcondition} and the fact that~$\nabla \vlh$ is piecewise constant, it holds that 
$$
\sum_{K_h^+\neq\emptyset}\|\nabla \vlh\|_{0,K_h^+}^2=\|\nabla \vlh\|_{\Om^+,h}^2=0,
$$
which implies  that $\vlh=0$ since $\vlh\in \VL$. Thus, $\|\nabla \vlh\|_{\Om^+,h}$ is a norm on $\VL$. Similarly, since $\tauh\in \VRT$ is piecewise linear, it follows from \eqref{quadcondition} that if $\|\tauh\|_{\Om^-,h}=0$, 
$$
\sum_{K_h^-\neq\emptyset}\|\tauh\|_{0,K_h^-}^2=\|\tauh\|_{\Om^-,h}^2=0,
$$
which implies that $\tauh=0.$
Then, $\|\tauh\|_{\Om^-,h}$ is a norm on $\VRT$. Furthermore, $\interleave (\tauh,\vlh)\interleave_{1,h}$ and $\interleave \vrh \interleave_{0,h}$ in \eqref{normh}  are norms in $\VRT\times \VL$ and $\VP$, respectively.

Note that $\vlh\in\VL$   and $\tauh\in \VRT$ are piecewise linear with constant $\tauh\cdot\bmn$ on $\G_{K,h}$ if element $K$ intersects with the interface $\G$. Since \eqref{quadcondition} holds for the quadrature formulas \eqref{quad11}-\eqref{quad13}, the bilinear forms in \eqref{abquadrature} can be written in an equivalent way
\begin{equation}\label{abquad2}
\begin{aligned}
a_h(\sigmah,\ulh;\tauh,\vlh)=& \sum_{K\in \Th}\frac{1}{\betar}(\sigmah, \tauh)_{K_h^-}
+ (\betal \nabla \ulh, \nabla \vlh)_{K_h^+}
+ \langle\sigmah\cdot \bmn, \vlh\rangle_{\G_{K,h}}-   \langle\tauh\cdot \bmn, \ulh\rangle_{\G_{K,h}},
\\
b_h(\urh;\tauh, \vlh)=&\sum_{K\in \Th} (\urh, \Div \tauh)_{K_h^-} ,
\end{aligned}
\end{equation}  
and the norms in \eqref{normh} can be rewritten as
\begin{equation}\label{normh2}
\begin{aligned}
\interleave (\tauh,\vlh)\interleave_h^2 =&\sum_{K\in \Th}\frac{1}{\betar}\|\tauh\|_{0,K_h^-}^2
+ \|\Div\tauh\|_{0,K_h^-}^2 
+ \betal \|\nabla \vlh\|_{0,K_h^+}^2,\quad 
 \interleave \vrh\interleave_h^2=&\sum_{K\in \Th}\|\vrh\|_{0,K_h^-}^2.
 \end{aligned}
\end{equation}  
By the  trace inequality, the Cauchy-Schwarz inequality and the Poincar${\rm\acute{e}}$ inequality that
$$
\begin{aligned}
\sum_{K\in\Th}\left|\langle \sigmah\cdot \bmn, \vlh \rangle_{\G_{K,h}}\right| 
\le&\sum_{K\in\Th}\|\sigmah\cdot \bmn\|_{-\frac12,\G_{K,h}}\|\vlh\|_{\frac12,\G_{K,h}}
\\
\lesssim&(\sum_{K\in\Th} \|\Div\sigmah\|_{0,K_h^-}^2 +  \|\sigmah\|_{0,K_h^-}^2)^\frac12
(\sum_{K\in\Th} \|\vlh\|_{1,K_h^+}^2 )^\frac12
\\
\lesssim &(\sum_{K\in\Th} \|\Div\sigmah\|_{0,K_h^-}^2 +  \|\sigmah\|_{0,K_h^-}^2)^\frac12 (\sum_{K\in\Th} \|\nabla\vlh\|_{0,K_h^+}^2)^\frac12,
\end{aligned}
$$ 
which completes the proof for the boundedness \eqref{bdd}. 
\end{proof}

In order to verify the inf-sup condition \eqref{infsupb} in Assumption \ref{ass:error}, we first design new interpolation operators, which admits the crucial commuting property on all elements in~$\cT_h^-$ under reasonable assumptions. 


\begin{figure}[!ht]
\begin{center}
\begin{tikzpicture}[xscale=4.3,yscale=3.5]
\node at (2.5,0.65) {$\bmp_1$};   
\node at (3,0.65) {$\bmp_2$};  
\node at (3.5,0.65) {$\bmp_3$};  
\node at (3.8,0.65) {$\bmp_4$};  
\node at (4.2,0.65) {$\bmp_5$};  
\node at (4.6,0.65) {$\bmp_6$};  
\node at (5,0.65) {$\bmp_7$};  
\node at (5.7,0.65) {$\bmp_8$}; 
\node at (6,0.65) {$\bmp_9$};  
\draw[-,ultra thick,red] (6,0.75) -- (6,1.5);		\node at (6.1,1.2) {$e_{9}^1$};  
\draw[-,ultra thick] (6,1.5) -- (2.5,1.5);
\draw[-,ultra thick,red] (2.5,0.75) -- (2.5,1.5);	\node at (2.45,1.2) {$e_{1}^1$};  
\node at (2.7,0.9) {$K_1^1$};  

\draw[dotted,ultra thick] (2.5,0.75) -- (3,0.75);
\draw[-,ultra thick,blue] (2.5,1.5) -- (3,0.75);	\node at (2.75,1.2) {$e_{2}^1$};  
\draw[dotted,ultra thick,blue] (2.8,1.5) -- (3,0.75);	\node at (2.92,1.2) {$e_{2}^2$};  
\draw[dotted,ultra thick,blue] (3.1,1.5) -- (3,0.75);	\node at (3.1,1.2) {$e_2^3$};  
\draw[dotted,ultra thick,blue] (3.6,1.5) -- (3,0.75);	\node at (3.3,1.2) {$e_2^4$};  
\node at (2.7,1.4) {$K_2^1$}; 
\node at (2.9,1.4) {$K_2^2$}; 
\node at (3.2,1.4) {$K_2^3$}; 

\draw[dotted,ultra thick] (3,0.75) -- (3.5,0.75);
\draw[-,ultra thick,red] (3.6,1.5) -- (3.5,0.75);	\node at (3.55,1.1) {$e_{3}^1$};  
\node at (3.4,0.9) {$K_2^4$}; 
\node at (3.65,0.9) {$K_3^1$};

\draw[dotted,ultra thick] (3.8,0.75) -- (3.5,0.75);
\draw[-,ultra thick,blue] (3.6,1.5) -- (3.8,0.75);	\node at (3.74,1.2) {$e_{4}^1$};  
\draw[dotted,ultra thick,blue] (4,1.5) -- (3.8,0.75);	\node at (3.95,1.2) {$e_{4}^2$};  
\draw[dotted,ultra thick,blue] (4.5,1.5) -- (3.8,0.75);	\node at (4.15,1.2) {$e_4^3$};  
\node at (3.8,1.4) {$K_4^1$}; 
\node at (4.1,1.4) {$K_4^2$}; 

\draw[dotted,ultra thick] (3.8,0.75) -- (4.2,0.75);
\draw[dotted,ultra thick] (4.2,0.75) -- (4.6,0.75);
\draw[-,ultra thick,red] (4.5,1.5) -- (4.2,0.75);	\node at (4.38,1.1) {$e_5^1$};  
\draw[-,ultra thick,red] (4.5,1.5) -- (4.6,0.75);	\node at (4.57,1.2) {$e_6^1$}; 
\node at (4.1,0.9) {$K_4^3$}; 
\node at (4.4,0.9) {$K_5^1$}; 
\node at (4.7,0.9) {$K_6^1$};

\draw[dotted,ultra thick] (4.6,0.75) -- (5,0.75); 
\draw[-,ultra thick,blue] (4.5,1.5) -- (5,0.75);	\node at (4.75,1.2) {$e_7^1$};  
\draw[dotted,ultra thick,blue] (4.8,1.5) -- (5,0.75);	\node at (4.95,1.2) {$e_7^2$};  
\draw[dotted,ultra thick,blue] (5.5,1.5) -- (5,0.75);	\node at (5.25,1.2) {$e_7^3$};  
\draw[dotted,ultra thick,blue] (5.8,1.5) -- (5,0.75);	\node at (5.43,1.2) {$e_7^4$};  
\node at (4.75,1.4) {$K_7^1$}; 
\node at (5.1,1.4) {$K_7^2$}; 
\node at (5.55,1.4) {$K_7^3$}; 

\draw[dotted,ultra thick] (5,0.75) -- (5.7,0.75);
\draw[-,ultra thick,blue] (5.8,1.5) -- (5.7,0.75);	\node at (5.7,1.2) {$e_8^1$};  
\draw[dotted,ultra thick,blue] (6,1.5) -- (5.7,0.75);	\node at (5.85,1.2) {$e_8^2$};  
\node at (5.85,1.4) {$K_8^1$}; 
\node at (5.5,0.9) {$K_7^4$}; 
\node at (5.9,0.9) {$K_8^2$}; 

\draw[dotted,ultra thick] (5.7,0.75) -- (6,0.75);
\draw[thick] (2.5,1) parabola bend (4.25,1.3) (6,1);
\node at (2.8,0.66) {$e_{1}^2$};   
\node at (3.25,0.66) {$e_{2}^5$};  
\node at (3.7,0.66) {$e_{3}^2$};  
\node at (4,0.66) {$e_{4}^4$};  
\node at (4.4,0.66) {$e_{5}^2$};  
\node at (4.8,0.66) {$e_{6}^2$};  
\node at (5.3,0.66) {$e_{7}^5$};  
\node at (5.85,0.66) {$e_{8}^3$};

\node at (2.65,1.6) {$\tilde e_{2}^1$};   
\node at (2.95,1.6) {$\tilde e_{2}^2$}; 
\node at (3.3,1.6) {$\tilde e_{2}^3$};   
\node at (3.8,1.6) {$\tilde e_{4}^1$};  
\node at (4.2,1.6) {$\tilde e_{4}^2$};  
\node at (4.65,1.6) {$\tilde e_{7}^1$};  
\node at (5.1,1.6) {$\tilde e_{7}^2$};  
\node at (5.6,1.6) {$\tilde e_{7}^3$};  
\node at (5.9,1.6) {$\tilde e_{8}^1$};

\node at (4,1.6) {$\Omr$};  
\node at (4,0.55) {$\Oml$};  
\end{tikzpicture}

\end{center}
\caption{Notations of vertices, edges and elements for a two-dimensional example.}
\label{fig:dof1}
\end{figure}
For any triangulation, the vertices in $\Oml $ of interface elements, denoted by $\bmp_1$,$\cdots$, $\bmp_\neout$, form a polyline. For each vertex $\bmp_\ell$ with $1\le\ell\le \neout$, denote the intersecting edges with endpoint $\bmp_\ell$ by $e_{\ell}^1$, $\cdots$, $e_{\ell}^{n_\ell}$, and the interface element with edges $e_{\ell}^j$ and $e_{\ell}^{j+1}$ by $K_\ell^j$. Denote the other interface element with edge $e_{\ell}^{n_\ell}$ by $K_\ell^{n_\ell}$, and the edge of $K_\ell^{n_\ell}$ in $\Oml$ by  $e_\ell^{n_\ell+1}$. For each interface element $K_\ell^j$ with $1\le j\le n_\ell-1$, denote the edge in $\Omr$ by $\tilde e_\ell^j$. An example of these notations is shown in Figure~\ref{fig:dof1}.  For intersecting edges, let $\bn_{e_\ell^i}$ be the unit normal direction of edge $e_\ell^i$ pointing from the element with smaller index to the element with larger index, and for interior edges $e$ to $\Omr$ of intersecting elements, let $\bn_{e}$ be the outer normal direction of the interface element. 
For each vertex $\bmp_\ell$ with $1\le\ell\le \neout$, define the set
$$
\mathcal{K}_\ell=\{K_\ell^j\}_{j=1}^{n_\ell}.
$$ 
Note that all the interface elements are exactly the union of all $\mathcal{K}_\ell$, where all the elements in each set share a common vertex in $\Oml$, and there exists at least one interior edge to $\Oml$ in each set.

Next we modify the definition of the canonical interpolation of the Raviart-Thomas element on interface elements to guarantee the important commuting property for the inf-sup condition. 
For any function $\taur\in H^1(\Omr)$, let $d_e(\taur)$ be the degrees of freedom of the Raviart-Thomas element with respect to edge $e$, namely
$
d_e(\taur) =
 \frac{1}{|e|}\int_{e}\taur\cdot \bn_{e} \ds,
$
and $\phi_e(x)$ be the corresponding basis function. Since the edge $e_\ell^i$ with $i\ge 2$ is a common edge of elements~$K_\ell^{i-1}$ and $K_\ell^{i}$,  and the corresponding basis functions 
\begin{equation}\label{basis}
\phi_{e_\ell^i}|_{K_\ell^{i-1}}=\frac{|e_\ell^i|}{2|K_\ell^{i-1}|}(x-p_{K_\ell^{i-1}}^i),\quad \phi_{e_\ell^i}|_{K_\ell^i}=-\frac{|e_\ell^i|}{2|K_\ell^i|}(x-p_{K_\ell^{i}}^{i}),
\end{equation}
where $p_K^i$ is the vertex of $K$ not belonging to $e_\ell^i$. 
Design an interpolation $\PiRTnew: H^1(\Omr)\rightarrow \VRT$ as 
\begin{equation}\label{interpolationdef}
|e|d_e(\PiRTnew\taur) = \begin{cases}
|e|d_e(\taur), &e=e_\ell^1,\ \mbox{or}\ e\subset \Omr
\\
|e_\ell^{i-1}|d_{e_\ell^{i-1}}(\taur) - d_{\tilde e
_\ell^{i-1}}(\taur)|\tilde e_\ell^{i-1}| + |K_\ell^{i-1}|\Pi_{K_\ell^{i-1}}^0\Div\taur ,& e=e_\ell^i,\ 2\le i\le n_\ell
\\
|e_\ell^{n_\ell}|d_{e_\ell^{n_\ell}}(\taur) - d_{e
_{\ell+1}^1}(\taur)|e
_{\ell+1}^1| + |K_\ell^{n_\ell}|\Pi_{K_\ell^{n_\ell}}^0\Div\taur,&e=e_\ell^{n_\ell+1}
\end{cases},
\end{equation}
where  the $L^2$ projection $\Pih \vr\in \Vhr$  satisfies that 
\begin{equation}\label{pi0}
\Pih \vr|_K=\Pi_K^0 \vr,\quad \mbox{with}\quad \Pi_K^0 \vr = \frac{1}{|K_h^-|}\int_{K_h^-} \vr\dx.
\end{equation} 

The lemma below shows that this new interpolation is bounded and admits the crucial commuting property and approximation property 
under the following assumption.

\begin{assumption}\label{ass:intersect}
Assume that the interface $\G$ is a $C^2$ curve.
\begin{enumerate}
\item The interface $\G$ can not intersect any edge at more than one point.
\item There exists a positive constant $c$ such that  
$$
ch^{4}\le |K_{h}^-|,\qquad \forall K\cap \G\neq \emptyset.
$$  
\end{enumerate}
\end{assumption} 
The first assumption above  holds if  the interface  is resolved enough by the unfitted mesh, and has been used in many works on unfitted meshes such as \cite{Adjerid2015AnID,chen2009the,gong2008immersed,gong2009immersed,Li1998TheII,li2006the,li2003new,2024A,guo2023Solving,burman2010fictitious,burman2012fictitious,hansbo2002an,nitsche1971ber,burman2018robust,wang2013hybridizable,casagrande2016DG,wang2014an,2022An,2024Unfitted}. The second assumption imposes a relatively loose restriction on  interface elements, namely the ratio~${|K_h^-|\over |K|}$ should be bounded below by $\mathcal{O}(h^2)$.

\begin{lemma}\label{lm:commuting}
If Assumption \ref{ass:intersect} holds, the interpolation $\PiRTnew: H^1(\Omr)\rightarrow \VRT$ is bounded, and admits the commuting property 
\begin{equation}\label{commuting2s}
\Div \PiRTnew \taur = \Pih \Div \taur.
\end{equation} 
Furthermore, there holds the inf-sup condition \eqref{infsupb} and the approximate property \eqref{app} with $k=1$, namely
\begin{equation}\label{interpolationerror}
\inf_{(\tauh, \vlh)\in \Qhr\times \Vhl}\interleave (\sigmar-\tauh,\ul-\vlh)\interleave_{1,h} 
+ \inf_{\vrh\in \Vhr}\interleave\ur-\vrh\interleave_{0,h}\lesssim h\|\taur\|_{1,\Omr}.
\end{equation}   
provided that $u\in H^2(\Oml\cup\Omr)\cap H^1(\Om)$.
\end{lemma}

\begin{proof}
Since $\PiRTnew\taur$ shares the same degrees of freedom with the canonical interpolation of the Raviart-Thomas element on interior edges to $\Omr$, the commuting property~\eqref{commuting2s} holds for any element $K\subset \Omr$. By \eqref{basis} and \eqref{interpolationdef}, the commuting property holds for $K_\ell^i$ with $2\le i\le n_\ell$ since
$$
\Div \PiRTnew\taur|_{K_\ell^i}= -d_{e_\ell^i}\frac{|e_\ell^i|}{|K_\ell^{i}|} +  d_{e_\ell^{i+1}}\frac{|e_\ell^{i+1}|}{|K_\ell^{i}|} + d_{\tilde e_\ell^i}\frac{|\tilde e_\ell^i|}{|K_\ell^{i}|}= \Pi_{K_\ell^{i}}^0\Div\taur.
$$
A similar argument proves the commuting property~\eqref{commuting2s} for all interface elements.

Consider  interface elements in a patch $\mathcal{K}_\ell$ with more than one interface element.  By~\eqref{interpolationdef}, it holds for $1\le \ell\le \neout$ and $1\le i\le n_\ell$ that
\begin{equation}\label{pinewi}
d_{e_\ell^i}(\PiRTnew \taur) = {|e_\ell^{1}|\over |e_\ell^i|}d_{e
_\ell^{1}}(\taur) + \sum_{j=1}^{i-1} {|K_\ell^{j}|\over |e_\ell^i|}\Pi_{K_\ell^{j}}^0\Div\taur - \sum_{j=1}^{i-1}{|\tilde e_\ell^{j}| \over |e_\ell^i|}d_{\tilde e
_\ell^{j}}(\taur) ,
\end{equation}
and  
$
d_{e_\ell^{n_\ell+1}}(\PiRTnew \taur) = {|e_\ell^{n_\ell}| \over |e_\ell^{n_\ell+1}|}d_{e_\ell^{n_\ell}}(\taur) - {|e
_{\ell+1}^1|\over |e_\ell^{n_\ell+1}|} d_{e_{\ell+1}^1}(\taur)+ {|K_\ell^{n_\ell}|\over |e_\ell|}\Pi_{K_\ell^{n_\ell}}^0\Div\taur,
$ namely,
$$
d_{e_\ell^{n_\ell+1}}(\PiRTnew \taur) = {|e_\ell^1| \over |e_\ell^{n_\ell+1}|}d_{e
_\ell^{1}}(\taur) - {|e_{\ell+1}^1| \over |e_\ell^{n_\ell+1}|}d_{e
_{\ell+1}^{1}}(\taur)  + \sum_{j=1}^{n_\ell} {|K_\ell^{j}|\over |e_\ell^{n_\ell+1}|}\Pi_{K_\ell^{j}}^0\Div\taur - \sum_{j=1}^{n_\ell-1}{|\tilde e_\ell^{j}| \over |e_\ell^{n_\ell+1}|}d_{\tilde e
_\ell^{j}}(\taur).
$$
Let $\tilde K_\ell^j$ be the interior element to $\Omr$ with edge $\tilde e_\ell^j$, and $\tilde{\mathcal{K}}_\ell=\cup_{j=1}^{n_\ell-1} \tilde K_\ell^j$, and $I_h\tau$ be the Crouzeix-Raviart interpolation of $\tau\in H^1(\Omr)$. By the interpolation error estimate and the scaling argument,
$$
\begin{aligned}
|d_{\tilde e_\ell^j}(\taur)| = &|\frac{1}{|\tilde e_\ell^j|}\int_{\tilde e_\ell^j}I_h\taur\cdot \bn_{\tilde e_\ell^j}\ds|
\lesssim |\tilde K_\ell^j|^{-\frac12}\|I_h\taur\|_{0,\tilde K_\ell^j} +  \|\nabla I_h\taur\|_{0,\tilde K_\ell^j}
\lesssim |\tilde K_\ell^j|^{-\frac12}\|\taur\|_{0,\tilde K_\ell^j} +  \|\nabla \taur\|_{0,\tilde K_\ell^j}.
\end{aligned}
$$   
Note that
$
{|K_\ell^{j}|\over |e_\ell^i|}| \Pi_{K_\ell^{j}}^0\Div\taur | \lesssim ({|K|\over |K_h^-|})^{\frac12}\|\Div\taur\|_{0,K_h^-}
$ with $K=K_\ell^{j}$. 
Since the underlying triangulation is regular, 
it follows from \eqref{interpolationdef} and \eqref{pinewi} that for $K=K_\ell^i$ with $1\le i\le n_\ell$,
\begin{equation*}
\begin{aligned}
\|\PiRTnew\taur\|_{0,K_h^-}
\lesssim & \left(|d_{e_\ell^{1}}(\taur)| +|d_{e_{\ell+1}^{1}}(\taur)| +
	\sum_{ \tilde K\in \tilde{\mathcal{K}}_\ell} | \tilde K|^{-\frac12}\|\taur\|_{0, \tilde K} +  \|\nabla\taur\|_{0, \tilde K} \right.
	\\
	&\left. + \sum_{ \tilde K\in \mathcal{K}_\ell}|\tilde K|^\frac12 |\tilde{K}_h^-|^{-\frac12}\|\Div\taur\|_{0,\tilde{K}_h^-}\right)|K_h^-|^\frac12.
\end{aligned}
\end{equation*} 
A summation of the square of the estimate above on all interface elements leads to
\begin{equation*}
\begin{aligned}
\sum_{n_\ell>1}\sum_{K\in \mathcal{K}_\ell}\|\PiRTnew\taur\|_{0,K_h^-}^2\lesssim & \sum_{n_\ell>1}\left(|d_{e_\ell^{1}}(\taur)|^2 +|d_{e_{\ell+1}^{1}}(\taur)|^2 +
	\sum_{ \tilde K\in \tilde{\mathcal{K}}_\ell} | \tilde K|^{-1}\|\taur\|^2_{0, \tilde K} +  \|\nabla\taur\|^2_{0, \tilde K} \right.
	\\
	&\left. + \sum_{ \tilde K\in \mathcal{K}_\ell}|\tilde K|  |\tilde{K}_h^-|^{-1}\|\Div\taur\|^2_{0,\tilde{K}_h^-}\right)|\mathcal{K}_\ell^-|,
\end{aligned}
\end{equation*} 
where $|\mathcal{K}_\ell^-|=\sum_{K\in \mathcal{K}_\ell} |K_h^-|$. 
Note that the number of elements in $\mathcal{K}_\ell$ is bounded above. The regular mesh implies that $|K|^{-1}|\mathcal{K}_\ell^-|\lesssim 1$ for any $ K\in \mathcal{K}_\ell\cup \tilde{\mathcal{K}}_\ell$. And by Assumption~\ref{ass:intersect}, $|K| |K_h^-|^{-1}|\mathcal{K}_\ell^-|\lesssim 1$ for any $K\in \mathcal{K}_\ell$.
Since
$
| d_{e
_\ell^{1}}(\taur)|\lesssim |K_\ell^1|^{-\frac12}\| \taur\|_{0,K_\ell^1} + \|\nabla \taur\|_{0,K_\ell^1},
$
a combination of the estimates above yields
\begin{equation}\label{bdd:interpolation}
\begin{aligned}
\sum_{n_\ell>1}\sum_{K\in \mathcal{K}_\ell}\|\PiRTnew\taur\|_{0,K_h^-}^2\lesssim & \sum_{n_\ell>1}
	\sum_{K\in \mathcal{K}_\ell\cup\tilde{\mathcal{K}}_\ell} \|\taur\|^2_{0, K} +  \|\nabla\taur\|^2_{0, K}.
\end{aligned}
\end{equation}   

For a patch $\mathcal{K}_\ell$ with only one interface element $K=K_\ell^1$, the degrees of freedom of $\PiRTnew\taur$ on the two intersecting edges $e_\ell^1$ and $e_{\ell+1}^1$ are the same as the canonical interpolation of the Raviart-Thomas element. A similar argument to the one for Theorem 2.5 in \cite{duran2008error} shows that
$$
\begin{aligned}
\|\taur - \PiRTnew\taur \|_{0,K}\le& {C\over \sin \alpha}\left(\|(\taur - \PiRTnew\taur)\cdot \bn_{e_\ell^1} \|_{0,K} + \|(\taur - \PiRTnew\taur)\cdot \bn_{e_{\ell+1}^1}\|_{0,K}\right)
\\
\lesssim &  h\left(\|\nabla(\taur - \PiRTnew\taur) \bn_{e_\ell^1} \|_{0,K} + \|\nabla(\taur - \PiRTnew\taur) \bn_{e_{\ell+1}^1}\|_{0,K}\right),
\end{aligned}
$$
where $\alpha$ is the angle between $e_{\ell}^1$ and $e_{\ell+1}^1$.
Thus,
$$
\|\PiRTnew\taur \|_{0,K}
\lesssim \|\taur\|_{0,K}+ h\|\nabla\taur\|_{0,K} +   h\|\nabla \PiRTnew\taur \|_{0,K}.
$$
Note that $\nabla \PiRTnew\taur=\frac12 (\Div\PiRTnew\taur) I$ is constant on $K$, which indicates that
$$
\|\nabla \PiRTnew\taur \|_{0,K}={|K|^\frac12\over |K_h^-|^\frac12}\|\nabla \PiRTnew\taur \|_{0,K_h^-}\lesssim {|K|^\frac12\over |K_h^-|^\frac12}\|\nabla \taur \|_{0,K_h^-}.
$$
It follows that
\begin{equation}\label{bdd:interpolation2}
\|\PiRTnew\taur \|_{0,K_h^-}\le \|\PiRTnew\taur \|_{0,K}
\lesssim \|\taur\|_{0,K}+ h\|\nabla\taur\|_{0,K} +   h {|K|^\frac12\over |K_h^-|^\frac12}\|\nabla \taur \|_{0,K_h^-}.
\end{equation}
By Assumption~\ref{ass:intersect}, $h {|K|^\frac12\over |K_h^-|^\frac12}\lesssim 1$. Thus, a combination of \eqref{bdd:interpolation} and \eqref{bdd:interpolation2} leads to 
$$
\sum_{K\cap \Omr\neq \emptyset}\|\PiRTnew\taur \|_{0,K_h^-}^2\lesssim \sum_{K\cap \Omr\neq \emptyset}\|\taur \|_{0,K}^2 + \|\nabla\taur \|_{0,K}^2\lesssim \sum_{K\cap \Omr\neq \emptyset}\|\taur \|_{0,K_h^-}^2 + \|\nabla\taur \|_{0,K_h^-}^2,
$$
where the second estimate comes from the Sobolev extension theorem. This, together with
$
\|\Div\PiRTnew\taur\|_{\Omr,h}=\|\Pi_h^0\Div\taur\|_{\Omr,h}\lesssim \|\Div\taur\|_{\Omr,h}
$
by the commuting property~\eqref{commuting2s},  proves the boundedness of the interpolation operator $\PiRTnew$.


For any $\vrh\in \Vhr$, let  $\vlh=0$ and $\tauh=\PiRTnew  \taur$, where 
$$
\Div \taur=\vrh\ \quad \mbox{on}\ \Omr\cup \Omrh,\quad \mbox{with}\quad \|\taur\|_{1,\Omr}\lesssim \interleave\vrh \interleave_{0,h}.
$$  
By the commuting property~\eqref{commuting2s} and boundedness of $\PiRTnew$,  
\begin{equation} \label{infsupbdd}
\interleave (\tauh,\vlh)\interleave_{1,h}\lesssim  \|\PiRTnew\taur\|_{\Omr,h} + \|\Div\PiRTnew\taur\|_{\Omr,h}\lesssim \interleave \vrh\interleave_{0,h},
\end{equation}  
which, together with the fact that
$
b_h(\vrh;\tauh, \vlh)
=  \interleave \vrh\interleave_{0,h} ^2,
$
leads to the discrete inf-sup condition \eqref{infsupb} of the bilinear form $b_h(\cdot,\cdot)$, namely, 
\begin{equation*} 
\inf_{0\neq \vrh\in \Vhr}\sup_{(\tauh,\vlh)\in \Qhr\times \Vhl}\frac{b_h(\vrh;\tauh, \vlh)}{\interleave (\tauh,\vlh)\interleave_{1,h}\interleave\vrh \interleave_{0,h}}\ge \alpha>0.
\end{equation*}  

Given any $u\in H^2(\Oml\cup\Omr)\cap H^1(\Om)$, let
$\tauh=\PiRT \sigma^-$, $ \vlh=\Pi_L u^+$ and $ \vrh=\Pi_h^0u^-$
be the canonical interpolation of the Raviart-Thomas element, the linear element and the piecewise constant projection, respectively. It holds that
$$
\begin{aligned}
&\|\sigmar-\tauh\|_{\Omr,h} + \|\Div(\sigmar-\tauh)\|_{\Omr,h} + \|\nabla (\ul-\vlh)\|_{\Oml,h} + \|\ur-\vrh\|_{\Omr,h}
\\
\lesssim &\|(I-\PiRT)\sigma^-\|_{\Om} + \|\Div (I-\PiRT)\sigma^-\|_{\Om} + \|\nabla(I- \Pi_L)u^+\|_{\Om} + \|(I-\Pi_h^0)u^-\|_{\Om}\lesssim h,
\end{aligned}
$$  
which completes the proof.

\end{proof}

The following lemma analyzes the consistency error terms in Assumption \ref{ass:error}.
\begin{lemma}\label{lm:quadrature}
Under Assumption \ref{ass:intersect}, the consistency error estimate \eqref{consistency} holds for the DiFEM \eqref{discrete} equipped with quadrature schemes  \eqref{quad11}-\eqref{quad13} in conforming finite spaces \eqref{conformspaces} with  \eqref{shapefunc}. 
\end{lemma}
\begin{proof}  
By the quadrature formula \eqref{quadcondition} and the analysis of Theorem 4.1.4 in \cite{ciarlet2002the},  
\begin{equation}\label{f1}
| (\fl, \vlh)_{K_h^+} - (\fl,\vlh)_{K_h^+,h}|\lesssim h^2\|\fl\|_{2,K_h^+}\|\vlh\|_{0,K_h^+},\quad \forall \vlh\in \VL.
\end{equation}
Let $K_h^*$ be the region enclosed by the interface $\G_K=\G\cap K$  and the approximation~$\G_{K,h}$. According to Assumption \ref{ass:intersect}, each curve is of class $C^2$, then  
$
|K_h^*|\lesssim h^3.
$
The polygon $K_h^+$ is an approximation to $K^+$. It holds that 
\begin{equation}\label{f2}
|(\fl, \vlh)_{K_h^+} -  (\fl, \vlh)_{K^+}|\le \int_{K_h^*}| \fl\vlh |\dx \lesssim h\|\fl\|_{2,K_h^+}\|\vlh\|_{K_h^+, h}.
\end{equation}
A combination of \eqref{f1} and \eqref{f2} leads to
\begin{equation} \label{fest}
|(\fl, \vlh)_{\Oml} - (\fl, \vlh)_{\Oml,h}|\le \sum_{K\cap \Oml\neq \emptyset} | (\fl, \vlh)_{K\cap \Oml} - (\fl,\vlh)_{K_h^+,h}|\lesssim h\|\vlh\|_{\Oml,h},
\end{equation} 
which leads to the estimate of the first term in the third inequality of \eqref{consistency}. A similar analysis proves the second term in the third inequality and also the second inequality of \eqref{consistency}.



By the definition of the bilinear form  \eqref{abdef}, the quadrature formula  \eqref{abquad2} and the integration by parts,
\begin{equation}\label{a1}
\begin{aligned}
&a(\sigmar,\ul;\tauh,\vlh) -a_h(\sigmar,\ul;\tauh,\vlh)
\\
=&\frac{1}{\betar}\left((\sigmar,\tauh)_{\Omr}  - (\sigmar,\tauh)_{\Omr,h} \right)
+ \left((\betal \nabla \ul, \nabla \vlh)_{\Oml} - (\betal \nabla \ul, \nabla \vlh)_{\Oml,h}\right)
\\
&- \left(\langle \tauh\cdot \bmn, \ul\rangle_{\G}  - \langle \tauh\cdot \bmn, \ul\rangle_{\G,h} \right)
+ \left( \langle \sigmar\cdot \bmn, \vlh\rangle_{\G}   -  \langle \sigmar\cdot \bmn, \vlh\rangle_{\G,h}  \right)
\end{aligned}
\end{equation}
For any element $K$ intersecting with $\G$, by a  direct application of the integration by parts and the fact that $|K_h^*|\lesssim h^3$, it holds that
$$
|\langle \tauh\cdot \bmn, \ul \rangle_{\G_K} - \langle \tauh\cdot \bmn, \ul \rangle_{\G_{K,h}} |= |\int_{K_h^*}\ul\Div\tauh + \tauh:\nabla \ul\dx|\lesssim h(\|\Div\tauh\|_{0,K_h^-} + \|\tauh\|_{0,K_h^-}).
$$
A summation of the estimate above on all intersecting elements gives  
$$
|\langle \tauh\cdot \bmn, \ul\rangle_{\G}  - \langle \tauh\cdot \bmn, \ul\rangle_{\G,h}|
\lesssim h \interleave (\tauh,\vlh)\interleave_{1,h}.
$$
A similar analysis leads to  $|\langle \sigmar\cdot \bmn, \vlh\rangle_{\G}  - \langle \sigmar\cdot \bmn, \vlh\rangle_{\G,h} |\lesssim h\interleave (\tauh,\vlh)\interleave_{1,h}$,  and also the last inequality of \eqref{consistency}. By a similar analysis of \eqref{fest}, the summation of the first four terms on the right-hand side of \eqref{a1} are also of $\mathcal{O}(h)$. A substitution of these estimates into \eqref{a1} yields that
\begin{equation}\label{a2}
|a(\sigmar,\ul;\tauh,\vlh) -a_h(\sigmar,\ul;\tauh,\vlh)|\lesssim h\interleave (\tauh,\vlh)\interleave_{1,h},
\end{equation} 
which proves the first inequality in \eqref{consistency}, and completes the proof. 
\end{proof}

 \begin{remark}
The consistency error estimate \eqref{consistency} require the pointwise value of $\sigmar$, $\taur$, $\ur$, $\ul$, $\vl$ and $\fr$ since quadrature schemes are used. By the Sobolev space embedding theorem, the pointwise value of a function is well defined if the function belongs to $H^{1+\varepsilon}(\Om^s)(\varepsilon>0)$, which is not the case for the lowest order DiFEM. However, thanks to \eqref{quadcondition}, the bilinear forms in \eqref{consistency} of functions in $H^1(\Om^s)$ can be defined as  equivalent inner products to the one computed by quadrature scheme. Thus, the consistency error estimate \eqref{consistency} is still well defined for functions in $H^1(\Om^s)$.
\end{remark}

A combination of the boundedness in Lemma~\ref{lm:bdd},  the inf-sup condition and the approximation property in Lemma~\ref{lm:commuting}, and the consistency error estimates in Lemma~\ref{lm:quadrature} leads to the optimal convergence of the lowest-order DiFEM in the following theorem.

\begin{theorem}\label{th:con}
If Assumption \ref{ass:intersect} holds, there exists a unique solution  $(\sigmah,\urh,\ulh)\in \VRT\times \VP\times \VL$ of the DiFEM \eqref{discrete} equipped with quadrature schemes \eqref{quad11}-\eqref{quad13}  in conforming finite spaces \eqref{conformspaces} with  \eqref{shapefunc}, and 
$$
\interleave (\sigmar-\sigmah,\ul-\ulh)\interleave_{1,h} + \interleave\ur-\urh\interleave_{0,h}\lesssim h,
$$
provided that $u\in H^2(\Oml\cup \Omr)\cap H^1(\Om)$.
\end{theorem}

\section{Numerical Examples}\label{sec:numerical}

This section presents several numerical tests to illustrate the performance of the proposed DiFEM.

\subsection{Example 1}
Let the domain $\Om$ be the square $(0,1)^2$, and the interface  $\G:=\{(x,y): \phi(x,y)=0.5 + 0.2\sin(\pi x)-y=0,\ 0<x<1\}$. Consider the problem
$$
\begin{aligned} 
&- \nabla \cdot(\beta\nabla u)= f & \mbox{ in }\Omega=\Oml\cup \Omr,\\ 
\label{jumpcon}
&[u]=0,\quad \left[\beta\nabla u\cdot\bmn\right]=0 & \mbox{ across }\Gamma,\\
\label{bdcon}
&u|_{\G_D}= g_D, \qquad  {\partial u\over \partial n}|_{\G_N}= g_N,   & 
\end{aligned}
$$
with $\G_D=\{(0,y): 0<y<1\}\cup \{(1,y): 0<y<1\}\cup \{(x,0): 0<x<1\}$ and $\G_N=\{(x,1):0<x<1\}$. Let $\Omr$ be the region above the interface, and $\Oml$ be the region below. 
The source term $f$ and the boundary conditions $g_D$ and $g_N$ are determined by the exact solution
$$
u=\left\{
\begin{aligned}
\cos(\pi x)\cos(\pi y)\phi(x,y)& \qquad \mbox{if}\ \phi(x,y)\le 0
\\
\betar \cos(\pi x)\cos(\pi y)\phi(x,y)&\qquad \mbox{if}\ \phi(x,y)>0
\end{aligned}\right.
$$
with $\betal=1$ and various $\betar$.
Let $\cT_0$ be the triangulation consisting of two right triangles obtained by cutting the unit square with a north-east line. Each triangulation $\cT_i$ is refined into a half-sized triangulation uniformly, to get a higher level triangulation~$\cT_{i+1}$.

As shown in Figure \ref{fig:ex1} for different $\betar/ \betal$, the relative errors  $\|\nabla(\ul-\ulh)\|_{0,\Omlh}$, $\|\sigmar-\sigmah\|_{0,\Omrh}$ and  $\|\ur-\urh\|_{0,\Omrh}$ of solutions by the DiFEM converge at the rate 1.00, and those of $\| \ul-\ulh\|_{0,\Omlh}$ converge at the rate 2.00, which coincide with the convergence result in Theorem \ref{th:con}. Note that the convergence rate does not deteriorate  even the ratios $\frac{\betar}{\betal}$ or $\frac{\betal}{\betar}$ are large, and verifies the  efficiency of the proposed DiFEM.


\begin{figure}[H]
\setlength{\abovecaptionskip}{0pt}
\setlength{\belowcaptionskip}{0pt}
\centering  
\includegraphics[width=6cm]{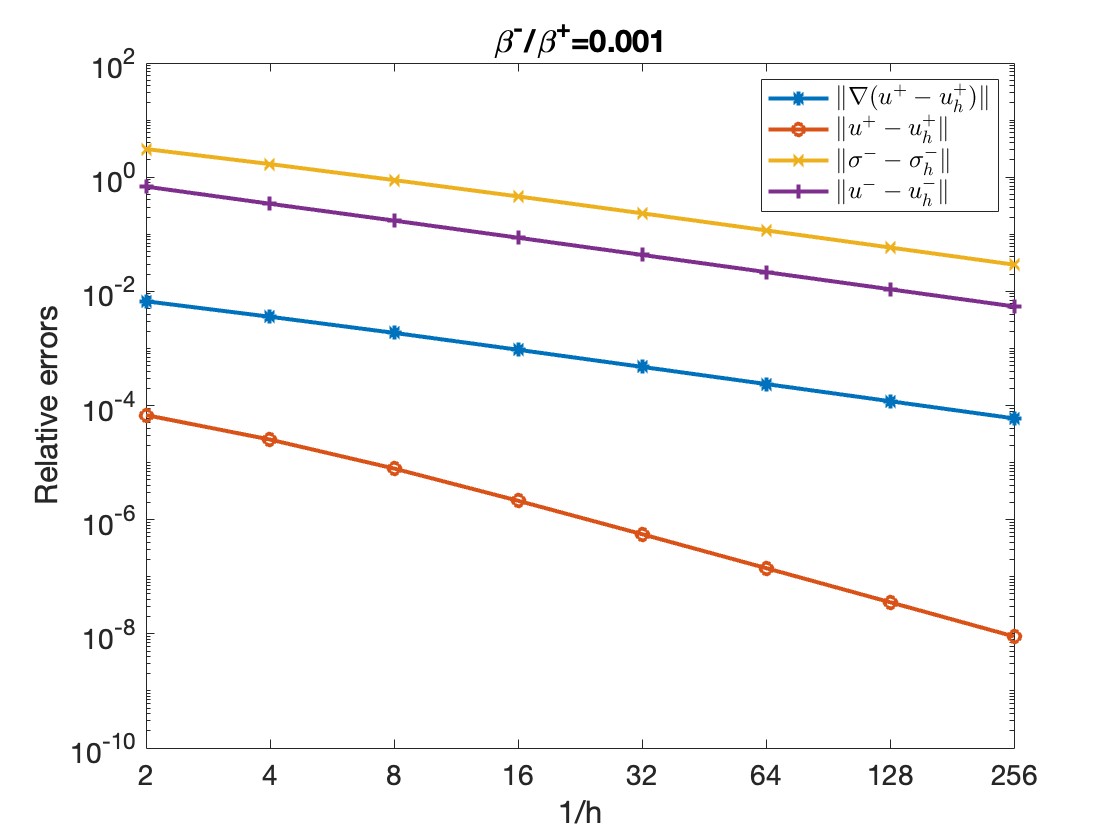}
\includegraphics[width=6cm]{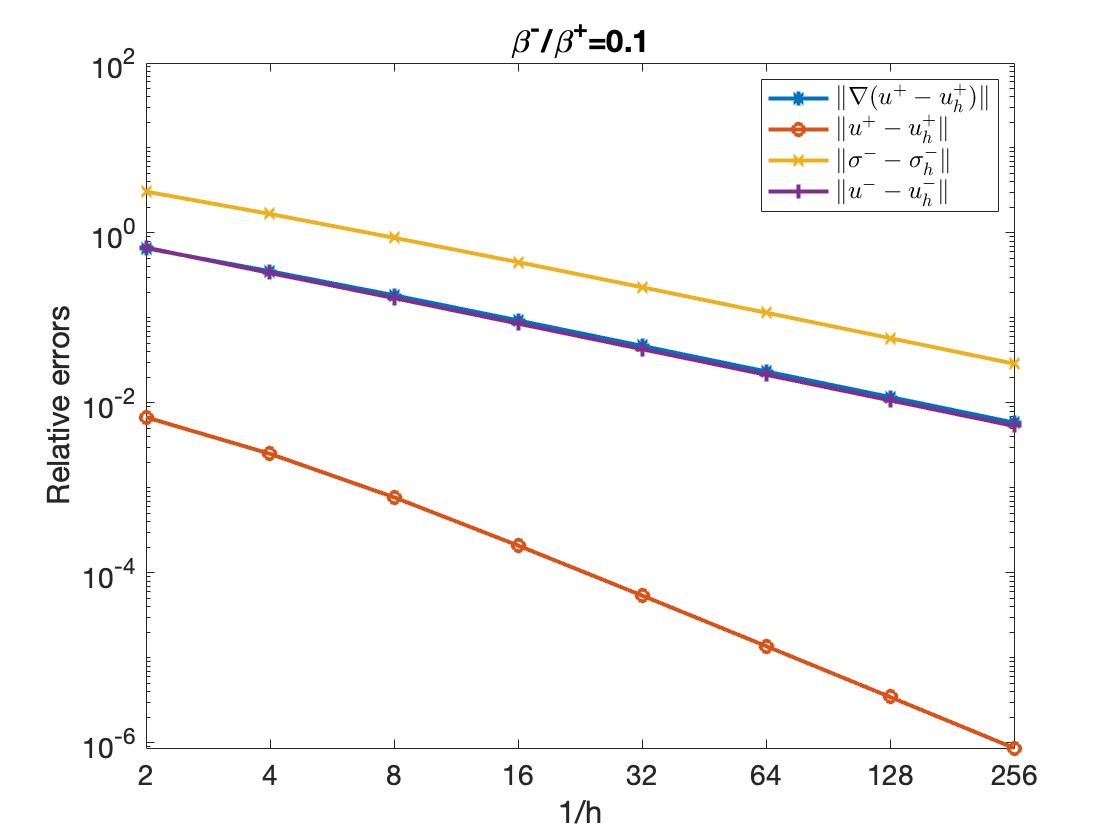}
\includegraphics[width=6cm]{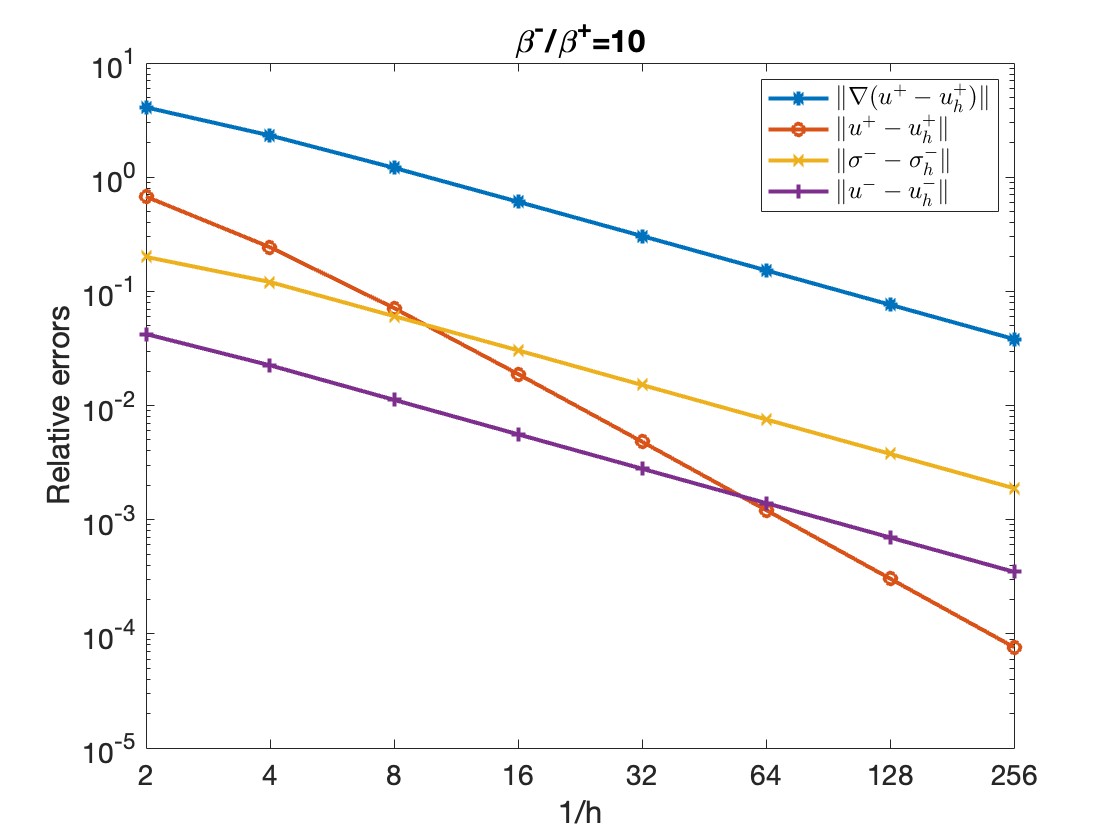}
\includegraphics[width=6cm]{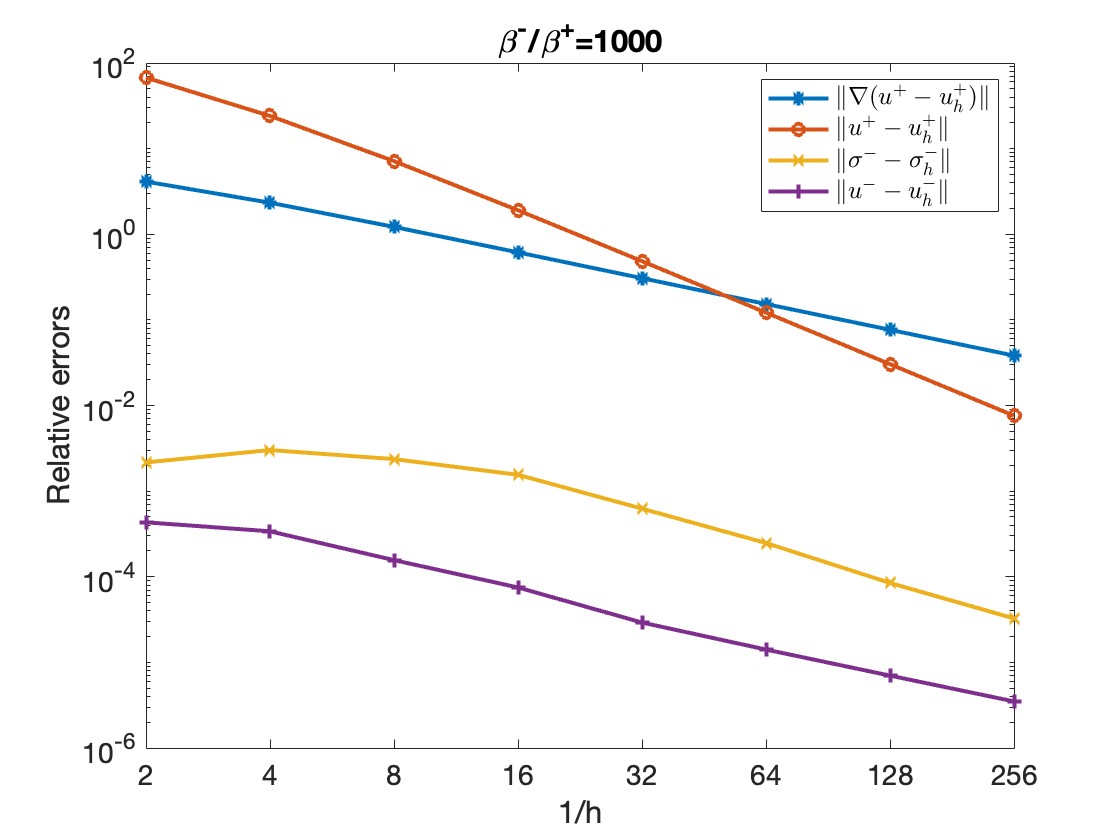}
\caption{\small Relative errors for Example 1 with different choice of $\frac{\betar}{\betal}$.}
\label{fig:ex1}
\end{figure}

\subsection{Example 2}
Let the domain $\Om$ be the square $(-2,2)^2$, and the interface be a circle centered at the origin $(0,0)$ with radius $r=1.1$. Let $\cT_0$ be the same triangulation in Example 1, and $\cT_2$ be the initial triangulation satisfying the requirement that the interface each edge at most once, which is depicted in Figure \ref{fig:circles}(a).
%
\begin{figure}[!ht]
\begin{center}
\begin{tikzpicture}[xscale=2,yscale=2]
\draw[-] (0,0) -- (2,0);
\draw[-] (2,0) -- (2,2);
\draw[-] (0,2) -- (2,2);
\draw[-] (0,0) -- (0,2);
\draw[-] (0,0) -- (2,2); 
\draw[-] (0,1) -- (1,2);
\draw[-] (1,0) -- (2,1);
\draw[-] (0,1) -- (2,1);
\draw[-] (1,0) -- (1,2);
\draw[-] (0,0.5) -- (1.5,2);
\draw[-] (0,1.5) -- (0.5,2);
\draw[-] (0.5,0) -- (2,1.5);
\draw[-] (1.5,0) -- (2,0.5);
\draw[-] (0.5,0) -- (0.5,2);
\draw[-] (1.5,0) -- (1.5,2);
\draw[-] (0,0.5) -- (2,0.5);
\draw[-] (0,1.5) -- (2,1.5);
\draw[line width=1.5pt] (1,1) circle (0.55); 
\end{tikzpicture}\qquad\qquad
\begin{tikzpicture}[xscale=2,yscale=2]
\draw[-] (0,0) -- (2,0);
\draw[-] (2,0) -- (2,2);
\draw[-] (0,2) -- (2,2);
\draw[-] (0,0) -- (0,2);
\draw[-] (0,0) -- (2,2); 
\draw[-] (0,1) -- (1,2);
\draw[-] (1,0) -- (2,1);
\draw[-] (0,1) -- (2,1);
\draw[-] (1,0) -- (1,2);
\draw[-] (0,0.5) -- (1.5,2);
\draw[-] (0,1.5) -- (0.5,2);
\draw[-] (0.5,0) -- (2,1.5);
\draw[-] (1.5,0) -- (2,0.5);
\draw[-] (0.5,0) -- (0.5,2);
\draw[-] (1.5,0) -- (1.5,2);
\draw[-] (0,0.5) -- (2,0.5);
\draw[-] (0,1.5) -- (2,1.5);
\draw[line width=1.5pt] (0.6,1) circle (0.38); 
\draw[line width=1.5pt] (1.4,1) circle (0.38); 
\end{tikzpicture}
\end{center} 
\caption{\footnotesize{The interfaces in Example 1(left) and Example 2
(right).}}
\label{fig:circles}
\end{figure}
Let $\Omr$ be the region enclosed by the circle, and $\Oml$ be the region outside the circle. Consider the interface problem \eqref{source} with $\betal=1$ and various $\betar$, and the right-hand side~$f$ and the boundary condition~$g$ are computed such that the exact solution 
$$
u=\left\{
\begin{aligned}
e^{x^2+y^2-r^2} + \betar r^2-1& \qquad \mbox{if}\ x^2+y^2\le r^2
\\
\betar (x^2+y^2)&\qquad \mbox{if}\ x^2+y^2> r^2
\end{aligned}\right.
$$
satisfying the continuity conditions in \eqref{source}. 

Table \ref{tab:1m3} - \ref{tab:1p3} record the relative errors 
$\|\nabla(\ul-\ulh)\|_{0,\Omlh}$, $\| \ul-\ulh\|_{0,\Omlh}$, $\|\sigmar-\sigmah\|_{0,\Omrh}$ and  $\|\ur-\urh\|_{0,\Omrh}$, and the convergence rate of solutions by the DiFEM when $\betar=10^{-3}$, $10^{-1}$, $1$, $10$ and $10^3$, respectively. These results coincide with the convergence result in Theorem \ref{th:con}.
It is pointed out in \cite{zhang2016coefficient} that it is inappropriate to use conforming finite element methods for large jump-coefficient problems because of the coefficient-dependent error estimate bound. For the DiFEM \eqref{discrete} coupling the conforming finite element method and the mixed finite element method, it shows surprisedly that the convergence rate of $\|\sigmar-\sigmah\|_{0,\Omrh}$ and   $\|\ur-\urh\|_{0,\Omrh}$ on $\Omr$ are even higher than one when the ratio $\frac{\betar}{\betal}$ is relatively large as indicated in Table \ref{tab:1p3}, while 
$\|\nabla(\ul-\ulh)\|_{0,\Omlh}$ remains of accuracy $\mathcal{O}(h)$. 

\begin{table}[!ht]
  \centering
  \caption{Relative errors and convergence rates for Example 2 with $\frac{\betar}{\betal}=0.001$.}
    \begin{tabular}{c|cc|cc|cc|cc}
\hline
          & $\|\nabla(\ul-\ulh)\|_{0,\Omlh}$ &   rate    & $\|\ul-\ulh\|_{0,\Omlh}$  & rate    &   $\|\sigmar-\sigmah\|_{0,\Omrh}$&  rate     &   $\|\ur-\urh\|_{0,\Omrh}$&   rate  \\\hline
 $\cT_2$ & 3.99E-03 &       & 1.09E-04 &       & 4.65E-01 &       & 1.01E+00 &  \\\hline
     $\cT_3$ & 1.43E-03 & 1.48  & 4.95E-04 & -2.18 & 2.93E-01 & 0.67  & 9.69E-01 & 0.05 \\\hline
     $\cT_4$ & 6.97E-04 & 1.04  & 3.86E-04 & 0.36  & 2.11E-01 & 0.47  & 4.39E-01 & 1.14 \\\hline
     $\cT_5$ & 3.31E-04 & 1.07  & 9.43E-05 & 2.03  & 1.10E-01 & 0.94  & 1.88E-01 & 1.22 \\\hline
     $\cT_6$ & 1.64E-04 & 1.01  & 2.35E-05 & 2.00  & 5.87E-02 & 0.91  & 8.00E-02 & 1.23 \\\hline
     $\cT_7$ & 8.20E-05 & 1.00  & 6.01E-06 & 1.97  & 2.98E-02 & 0.98  & 3.76E-02 & 1.09 \\\hline
    \end{tabular}%
  \label{tab:1m3}%
\end{table}%

\begin{table}[!ht]
  \centering
  \caption{Relative errors and convergence rates for Example 2 with $\frac{\betar}{\betal}=0.1$.}
    \begin{tabular}{c|cc|cc|cc|cc}
\hline
          & $\|\nabla(\ul-\ulh)\|_{0,\Omlh}$ &   rate    & $\|\ul-\ulh\|_{0,\Omlh}$  & rate    &   $\|\sigmar-\sigmah\|_{0,\Omrh}$&  rate     &   $\|\ur-\urh\|_{0,\Omrh}$&   rate  \\\hline
     $\cT_2$ & 3.10E-01 &       & 1.09E-02 &       & 3.40E-01 &       & 1.00E+00 &  \\\hline
     $\cT_3$ & 1.21E-01 & 1.35  & 4.47E-02 & -2.04 & 2.47E-01 & 0.46  & 8.74E-01 & 0.20 \\\hline
     $\cT_4$ & 5.93E-02 & 1.03  & 2.48E-02 & 0.85  & 1.79E-01 & 0.47  & 2.80E-01 & 1.64 \\\hline
     $\cT_5$ & 2.85E-02 & 1.06  & 5.92E-03 & 2.06  & 9.51E-02 & 0.91  & 1.18E-01 & 1.25 \\\hline
     $\cT_6$ & 1.42E-02 & 1.01  & 1.47E-03 & 2.01  & 5.07E-02 & 0.91  & 5.03E-02 & 1.23 \\\hline
     $\cT_7$ & 7.08E-03 & 1.00  & 3.77E-04 & 1.97  & 2.58E-02 & 0.98  & 2.37E-02 & 1.09 \\\hline
    \end{tabular}%
  \label{tab:1m1}%
\end{table}%
          
\begin{table}[!ht]
  \centering
  \caption{Relative errors and convergence rates for Example 2 with $\frac{\betar}{\betal}=1$.}
    \begin{tabular}{c|cc|cc|cc|cc}
\hline
          & $\|\nabla(\ul-\ulh)\|_{0,\Omlh}$ &   rate    & $\|\ul-\ulh\|_{0,\Omlh}$  & rate    &   $\|\sigmar-\sigmah\|_{0,\Omrh}$&  rate     &   $\|\ur-\urh\|_{0,\Omrh}$&   rate  \\\hline
     $\cT_2$ & 2.24E-01 &       & 9.77E-02 &       & 4.75E-02 &       & 1.88E-01 &  \\\hline
     $\cT_3$ & 1.12E-01 & 1.00  & 2.88E-02 & 1.76  & 3.34E-02 & 0.51  & 3.16E-02 & 2.57 \\\hline
     $\cT_4$ & 5.55E-02 & 1.01  & 6.71E-03 & 2.10  & 1.86E-02 & 0.85  & 1.31E-02 & 1.28 \\\hline
     $\cT_5$ & 2.77E-02 & 1.00  & 1.66E-03 & 2.02  & 9.92E-03 & 0.90  & 5.66E-03 & 1.21 \\\hline
     $\cT_6$ & 1.38E-02 & 1.00  & 4.24E-04 & 1.97  & 5.04E-03 & 0.98  & 2.68E-03 & 1.08 \\\hline
    \end{tabular}%
  \label{tab:1p0}%
\end{table}%

\begin{table}[!ht]
  \centering
  \caption{Relative errors and convergence rates for Example 2 with $\frac{\betar}{\betal}=10$.}
    \begin{tabular}{c|cc|cc|cc|cc}
\hline
          & $\|\nabla(\ul-\ulh)\|_{0,\Omlh}$ &   rate    & $\|\ul-\ulh\|_{0,\Omlh}$  & rate    &   $\|\sigmar-\sigmah\|_{0,\Omrh}$&  rate     &   $\|\ur-\urh\|_{0,\Omrh}$&   rate  \\\hline
$\cT_2$ & 4.96E-01 &       & 3.47E-01 &       & 5.73E-02 &       & 3.05E-01 &  \\\hline
     $\cT_3$ & 2.26E-01 & 1.13  & 9.88E-02 & 1.81  & 1.12E-02 & 2.36  & 3.51E-02 & 3.12 \\\hline
     $\cT_4$ & 1.13E-01 & 1.00  & 2.87E-02 & 1.78  & 3.72E-03 & 1.59  & 9.40E-03 & 1.90 \\\hline
     $\cT_5$ & 5.63E-02 & 1.01  & 6.65E-03 & 2.11  & 2.16E-03 & 0.78  & 2.26E-03 & 2.06 \\\hline
     $\cT_6$ & 2.81E-02 & 1.00  & 1.63E-03 & 2.02  & 1.07E-03 & 1.01  & 6.93E-04 & 1.70 \\\hline
     $\cT_7$ & 1.40E-02 & 1.00  & 4.17E-04 & 1.97  & 5.20E-04 & 1.05  & 2.85E-04 & 1.28 \\\hline
    \end{tabular}%
   \label{tab:1p1}%
\end{table}%

\begin{table}[!ht]
  \centering
  \caption{Relative errors and convergence rates for Example 2 with $\frac{\betar}{\betal}=1000$.}
    \begin{tabular}{c|cc|cc|cc|cc}
\hline
          & $\|\nabla(\ul-\ulh)\|_{0,\Omlh}$ &   rate    & $\|\ul-\ulh\|_{0,\Omlh}$  & rate    &   $\|\sigmar-\sigmah\|_{0,\Omrh}$&  rate     &   $\|\ur-\urh\|_{0,\Omrh}$&   rate  \\\hline
     $\cT_2$ & 8.56E-01 &       & 2.70E-01 &       & 7.81E-03 &       & 1.67E-01 &  \\\hline
     $\cT_3$ & 2.38E-01 & 1.85  & 9.65E-02 & 1.48  & 1.25E-03 & 2.65  & 4.25E-02 & 1.98 \\\hline
     $\cT_4$ & 1.14E-01 & 1.06  & 2.89E-02 & 1.74  & 1.95E-04 & 2.68  & 9.62E-03 & 2.15 \\\hline
     $\cT_5$ & 5.74E-02 & 0.99  & 7.01E-03 & 2.05  & 5.78E-04 & -1.57 & 3.02E-03 & 1.67 \\\hline
     $\cT_6$ & 2.84E-02 & 1.02  & 1.66E-03 & 2.08  & 9.84E-05 & 2.55  & 5.55E-04 & 2.44 \\\hline
     $\cT_7$ & 1.41E-02 & 1.01  & 4.18E-04 & 1.99  & 3.47E-05 & 1.50  & 1.32E-04 & 2.07 \\\hline
    \end{tabular}%
   \label{tab:1p3}%
\end{table}%

\subsection{Example 3} 
This example tests the effectivity of the proposed DiFEM for  \eqref{discrete} on the unit square $(0,1)^2$ with multiple interfaces as depicted in Figure \ref{fig:circles}(right). In this case, the interface is the union of two closely located circles with radius $r=0.19$ and centers $(0.3,0.5)$ and $(0.7,0.5)$, respectively. Let $\Omr$ be the region enclosed by the two circles and $\Oml$ be the region outside the circles.
Compute the right-hand side $f$ and the boundary condition $g$ with $\betal=1$ such that the exact solution 
$$
u=\left\{
\begin{aligned}
\frac{1}{\betar}\phi(x)&\qquad \mbox{if}\ x\in \Omr
\\
\phi(x)& \qquad \mbox{if }\ x\in \Oml
\end{aligned}\right. ,
$$
where $\phi(x)=((x-0.3)^2+(y-0.5)^2-r^2)((x-0.7)^2+(y-0.5)^2-r^2)$.
Take the same triangulation as in Example 1 and let $\cT_2$ be the initial triangulation.

Table \ref{tab:2p0} - \ref{tab:2p2} record the relative errors $\|\nabla(\ul-\ulh)\|_{0,\Omlh}$, $\| \ul-\ulh\|_{0,\Omlh}$, $\|\sigmar-\sigmah\|_{0,\Omrh}$ and  $\|\ur-\urh\|_{0,\Omrh}$, and the convergence rate of solutions by the DiFEM when $\betar=1$ and $100$, respectively. It shows that the proposed DiFEM is effective even when the interface is the union of closely located curves. The results in Table \ref{tab:2p0} - \ref{tab:2p2} coincide with the convergence result in Theorem \ref{th:con}, and the convergence rate does not deteriorate  even the ratios $\frac{\betar}{\betal}$ or $\frac{\betal}{\betar}$ are large.

\begin{table}[!ht]
  \centering
  \caption{Relative errors and convergence rates for Example 3 with $\frac{\betar}{\betal}=1$.}
    \begin{tabular}{c|cc|cc|cc|cc}
\hline
          & $\|\nabla(\ul-\ulh)\|_{0,\Omlh}$ &   rate    & $\|\ul-\ulh\|_{0,\Omlh}$  & rate    &   $\|\sigmar-\sigmah\|_{0,\Omrh}$&  rate     &   $\|\ur-\urh\|_{0,\Omrh}$&   rate  \\\hline
     $\cT_2$ & 4.66E-01 &       & 3.16E-01 &       & 6.16E-02 &       & 7.30E-02 &  \\\hline
     $\cT_3$ & 2.44E-01 & 0.94  & 8.45E-02 & 1.90  & 1.86E-02 & 1.72  & 2.62E-02 & 1.48 \\\hline
     $\cT_4$ & 1.23E-01 & 0.98  & 2.14E-02 & 1.98  & 8.90E-03 & 1.07  & 9.26E-03 & 1.50 \\\hline
     $\cT_5$ & 6.18E-02 & 1.00  & 5.38E-03 & 1.99  & 4.53E-03 & 0.97  & 3.95E-03 & 1.23 \\\hline
     $\cT_6$ & 3.09E-02 & 1.00  & 1.35E-03 & 2.00  & 2.31E-03 & 0.97  & 1.81E-03 & 1.13 \\\hline
    \end{tabular}%
  \label{tab:2p0}%
\end{table}%
          
\begin{table}[!ht]
  \centering
  \caption{Relative errors and convergence rates for Example 3 with $\frac{\betar}{\betal}=100$.}
    \begin{tabular}{c|cc|cc|cc|cc}
\hline
          & $\|\nabla(\ul-\ulh)\|_{0,\Omlh}$ &   rate    & $\|\ul-\ulh\|_{0,\Omlh}$  & rate    &   $\|\sigmar-\sigmah\|_{0,\Omrh}$&  rate     &   $\|\ur-\urh\|_{0,\Omrh}$&   rate  \\\hline
     $\cT_2$ & 5.01E-01 &       & 3.95E-01 &       & 7.35E-03 &       & 5.06E-02 &  \\\hline
     $\cT_3$ & 2.45E-01 & 1.03  & 8.48E-02 & 2.22  & 2.44E-03 & 1.59  & 2.25E-03 & 4.49 \\\hline
     $\cT_4$ & 1.23E-01 & 0.99  & 2.16E-02 & 1.98  & 4.55E-04 & 2.42  & 1.43E-03 & 0.65 \\\hline
     $\cT_5$ & 6.19E-02 & 1.00  & 5.41E-03 & 2.00  & 1.27E-04 & 1.84  & 4.72E-04 & 1.60 \\\hline
     $\cT_6$ & 3.10E-02 & 1.00  & 1.35E-03 & 2.00  & 5.61E-05 & 1.18  & 1.34E-04 & 1.82 \\\hline
    \end{tabular}%
  \label{tab:2p2}%
\end{table}%

\subsection{Example 4}

This example tests the effectivity of the proposed DiFEM for  \eqref{discrete} on the square~$(-1,1)^2$ with the interface being a flower, where the corresponding level set function is defined as 
$$
r=\frac12 - 2^{\sin (5\theta)-3}.
$$
Let $\Oml=\{(r,\theta): r>\frac12 - 2^{\sin (5\theta)-3}\}$ and $\Omr=\{(r,\theta): r<\frac12 - 2^{\sin (5\theta)-3}\}$.
Compute the right-hand side $f$ and the boundary condition $g$ such that the exact solution 
\begin{equation}\label{exactu}
u(r,\theta)= \left\{
\begin{aligned}
\frac{1}{\betar}r^2(r - \frac12 + 2^{\sin (5\theta)-3})&\qquad \mbox{if}\ (r,\theta)\in \Omr,
\\
\frac{1}{\betal}r^2(r - \frac12 + 2^{\sin (5\theta)-3})& \qquad \mbox{if }\ (r,\theta)\in \Oml
\end{aligned}\right. ,
\end{equation} 
Take the same triangulation as in Example 1 and let $\cT_1$ be the initial triangulation.

\begin{figure}[!ht]
\begin{center}
\begin{tikzpicture}[xscale=2,yscale=2]
\draw[-] (-1,-1) -- (1,-1);
\draw[-] (1,-1) -- (1,1);
\draw[-] (1,1) -- (-1,1);
\draw[-] (-1,-1) -- (-1,1);
\draw[-] (-1,0) -- (1,0);
\draw[-] (0,-1) -- (0,1);
\draw[-] (-1,-1) -- (1,1);
\draw[-] (-1,0) -- (0,1);
\draw[-] (0,-1) -- (1,0);
\draw[domain=0:360,samples=500,line width=1.5pt] plot (\x:{0.5-2^(sin(5*\x)-3)});
\end{tikzpicture}
\end{center}
\caption{The interface for Example 3. }
\label{fig:flower}
\end{figure}

As shown in Figure \ref{fig:flower}, the region enclosed by the interface is no longer convex. 
Table \ref{tab:3p0} - \ref{tab:3p1} record the relative errors $\|\nabla(\ul-\ulh)\|_{0,\Omlh}$, $\| \ul-\ulh\|_{0,\Omlh}$, $\|\sigmar-\sigmah\|_{0,\Omrh}$ and  $\|\ur-\urh\|_{0,\Omrh}$, and the convergence rate of solutions by the direct  finite element method when $\betar=1$ and $10$, respectively. The results in Table \ref{tab:3p0} - \ref{tab:3p1} also verify the convergence result in Theorem \ref{th:con}.

\begin{table}[!ht]
  \centering
  \caption{Relative errors and convergence rates for Example 3 with $\frac{\betar}{\betal}=1$.}
    \begin{tabular}{c|cc|cc|cc|cc}
\hline
          & $\|\nabla(\ul-\ulh)\|_{0,\Omlh}$ &   rate    & $\|\ul-\ulh\|_{0,\Omlh}$  & rate    &   $\|\sigmar-\sigmah\|_{0,\Omrh}$&  rate     &   $\|\ur-\urh\|_{0,\Omrh}$&   rate  \\\hline
     $\cT_1$ & 3.89E-01 &       & 2.27E-01 &       & 2.24E-02 &       & 2.31E-02 &  \\\hline
     $\cT_2$ & 2.02E-01 & 0.95  & 5.92E-02 & 1.94  & 1.76E-02 & 0.35  & 1.04E-02 & 1.16 \\\hline
     $\cT_3$ & 1.03E-01 & 0.97  & 1.60E-02 & 1.89  & 9.09E-03 & 0.95  & 2.66E-03 & 1.96 \\\hline
     $\cT_4$ & 5.17E-02 & 0.99  & 3.84E-03 & 2.06  & 4.50E-03 & 1.01  & 1.24E-03 & 1.10 \\\hline
     $\cT_5$ & 2.59E-02 & 1.00  & 9.63E-04 & 1.99  & 2.22E-03 & 1.02  & 5.16E-04 & 1.26 \\\hline 
    \end{tabular}%
  \label{tab:3p0}%
\end{table}%

\begin{table}[!ht]
  \centering
  \caption{Relative errors and convergence rates for Example 3 with $\frac{\betar}{\betal}=10$.}
    \begin{tabular}{c|cc|cc|cc|cc}
\hline
          & $\|\nabla(\ul-\ulh)\|_{0,\Omlh}$ &   rate    & $\|\ul-\ulh\|_{0,\Omlh}$  & rate    &   $\|\sigmar-\sigmah\|_{0,\Omrh}$&  rate     &   $\|\ur-\urh\|_{0,\Omrh}$&   rate  \\\hline
     $\cT_1$ & 3.89E-01 &       & 2.26E-01 &       & 1.13E-02 &       & 1.19E-02 &  \\\hline
     $\cT_2$ & 2.02E-01 & 0.95  & 5.87E-02 & 1.95  & 5.80E-03 & 0.97  & 2.24E-03 & 2.41 \\\hline
     $\cT_3$ & 1.03E-01 & 0.97  & 1.56E-02 & 1.91  & 1.24E-03 & 2.23  & 3.43E-04 & 2.71 \\\hline
     $\cT_4$ & 5.17E-02 & 0.99  & 3.83E-03 & 2.02  & 5.26E-04 & 1.23  & 2.55E-04 & 0.43 \\\hline
     $\cT_5$ & 2.59E-02 & 1.00  & 9.61E-04 & 1.99  & 2.41E-04 & 1.13  & 7.86E-05 & 1.70 \\\hline
    \end{tabular}%
  \label{tab:3p1}%
\end{table}%


\bibliographystyle{plain} 
\bibliography{ref}
\end{document}